\documentclass{IEEEtran}

\IEEEoverridecommandlockouts                              

 \usepackage{cite}

\usepackage{amsmath,amssymb,amsfonts,amsthm}
\usepackage{graphicx}
\usepackage{algorithm,algorithmic}
\usepackage{textcomp}

\usepackage{color}

\theoremstyle{plain}
\newtheorem{theorem}{Theorem}[section]
\newtheorem*{theorem*}{Theorem}

\newtheorem{lemma}[theorem]{Lemma}
\newtheorem*{lemma*}{Lemma}

\newtheorem*{corollary*}{Corollary}
\theoremstyle{definition}

\newtheorem{assumption}[theorem]{Assumption}
\theoremstyle{remark}
\newtheorem{remark}[theorem]{Remark}

\usepackage{makecell} 
\usepackage{tablefootnote}
\usepackage{float}
\usepackage{graphicx}
\usepackage{adjustbox}
\usepackage{makecell} 
\usepackage{enumitem}
\usepackage{booktabs} 
\usepackage{mathtools}
\usepackage{subfigure}

\newcommand{\be}{\begin{equation}}
\newcommand{\ee}{\end{equation}}
\newcommand{\ba}{\begin{array}}
\newcommand{\ea}{\end{array}}
\newcommand{\bad}{\begin{aligned}}
\newcommand{\ead}{\end{aligned}}
\def\norm#1{\|{#1}\|}
\usepackage{stfloats}

\title{\LARGE \bf
Decentralized Online Riemannian Optimization \\Beyond Hadamard Manifolds
}

\author{Emre Sahinoglu and Shahin Shahrampour
\thanks{This work is supported in part by NSF Award ECCS-2240788   as well as NSF CAREER Award ECCS-2442321.}
\thanks{E. Sahinoglu and S. Shahrampour are with the Department of Mechanical \& Industrial Engineering at Northeastern University, Boston, MA 02115, USA. 
        {\tt\small emails:\{sahinoglu.m, s.shahrampour\}@northeastern.edu}.}
}

\newcommand{\un}{\texttt{unif}}

\newcommand{\man}{\mathcal{M}}
\newcommand{\metg}{\mathfrak{g}}
\newcommand{\tansp}[1]{T_{#1} \man}

\newcommand{\ptg}[2]{P_{#1,#2}}

\newcommand{\rgrad}[1]{\mathrm{grad} #1}

\newcommand{\expm}[1]{\mathrm{Exp}_{#1}}
\newcommand{\sptansp}[2]{\mathbb{S}_{\tansp{#1}}(#2)}
\newcommand{\balltansp}[2]{\mathbb{B}_{\tansp{#1}}(#2)}
\newcommand{\projman}[1]{\mathrm{P}_{#1}}
\newcommand{\mansubset}{\mathcal{X}}
\newcommand{\logm}[1]{\mathrm{Log}_{#1}}
\newcommand{\shrinkset}{(1-\tau)\mathcal{X}}
\newcommand{\ballgeod}[2]{B_{#2}(#1)}

\newcommand{\innerprod}[2]{\langle #1 , #2 \rangle}

\newcommand{\ddt}{\frac{d}{dt}}

\newcommand{\Dds}{\frac{D}{ds}}

\newcommand{\fdel}{f^{\delta}}
\newcommand{\gdel}{g^{\delta}}

\newcommand{\sumregret}{\frac{1}{n}\sum_{i=1}^n \sum_{t=1}^T}

\newcommand{\sumn}[1]{\sum_{#1=1}^n}
\newcommand{\sumt}{\sum_{t=1}^T}

\newcommand{\kmax}{K_{\max}}
\newcommand{\kmin}{K_{\min}}

\newcommand{\regstafull}{Reg_T^{Full} }
\newcommand{\regstafulleq}{\sumregret f_t(x_{i,t}) -  \sumt f_t(x^*) }

\newcommand{\regstaban}{Reg_T^{2Ban} }
\newcommand{\regstabaneq}{\sumregret \frac{1}{2}(f_t(x_{i,t,1})+f_t(x_{i,t,2})) - \sumt f_t(x^*) }

\newcommand{\varset}[1]{\mathrm{Var}(\{#1\})}

\newcommand{\rinj}{r_\mathrm{inj}}
\newcommand{\rcx}{r_\mathrm{cx}}

\newcommand{\eg}{\mathbb{E}_u [\gdel(x)]}

\begin{document}

\maketitle
\thispagestyle{empty}
\pagestyle{empty}


\begin{abstract}
We study decentralized online Riemannian optimization over manifolds with possibly positive curvature, going beyond the Hadamard manifold setting. Decentralized optimization techniques rely on a consensus step that is well understood in Euclidean spaces because of their linearity. However, in positively curved Riemannian spaces, a main technical challenge is that geodesic distances may not induce a globally convex structure. In this work, we first analyze a curvature-aware Riemannian consensus step that enables a linear convergence beyond Hadamard manifolds. Building on this step, we establish a $O(\sqrt{T})$ regret bound for the decentralized online Riemannian gradient descent algorithm. Then, we investigate the two-point bandit feedback setup, where we employ computationally efficient gradient estimators using smoothing techniques, and we demonstrate the same $O(\sqrt{T})$ regret bound through the subconvexity analysis of smoothed objectives.
\end{abstract}

\section{Introduction}

Online optimization is a foundational framework in machine learning and decision-making, where a learner sequentially selects decisions in response to a stream of data, aiming to minimize cumulative loss over time \cite{cesa2006prediction,hazan2016introduction}. 
While online algorithms are well established in Euclidean spaces, the growing need to optimize over structured data domains—such as the Stiefel manifold in low-rank matrix recovery or the manifold of positive definite matrices in metric learning—has spurred interest in extending online methods to non-Euclidean geometries. This has led to the emergence of online \textit{Riemannian} optimization, which generalizes Euclidean techniques to curved spaces while preserving their adaptive and sequential nature \cite{hu2023minimizing,hu2023riemannian,maass2022tracking,wang2023online}.

On the other hand, {\it decentralization} is crucial in distributed learning environments, where data is dispersed across multiple agents, and centralized coordination is often infeasible due to privacy constraints or communication bottlenecks. Recent advances have extended decentralized offline optimization techniques to the Riemannian setting, allowing agents to perform updates directly on the manifold without projecting into Euclidean space, thereby addressing the challenges posed by non-Euclidean and potentially non-convex geometries \cite{chen2021decentralized,chen2023local,kraisler2023consensus,kraisler2023distributed}. Moreover, the development of online learning frameworks in decentralized Riemannian optimization is increasingly critical in dynamic environments, where data streams continuously and agents must adapt in real time with limited inter-agent communication. Motivated by this need, decentralized online Riemannian optimization was recently initiated by \cite{chen2024decentralized} in the context of Hadamard manifolds. However, theoretical guarantees for {\it positively curved spaces} and the {\it bandit feedback} setting remain open and largely unexplored.

In this paper, we study decentralized online Riemannian optimization over (potentially) positively curved manifolds, going beyond Hadamard manifolds. Specifically, we consider a network of $n$ agents collaboratively minimizing a global objective, $f_t(\cdot) = \frac{1}{n} \sum_{i=1}^n f_{i,t}(\cdot)$ at time $t$. Each agent $i$ only has access to its local objective function $f_{i,t}$ and can communicate with its neighbors. Our goal is to minimize the static regret under both full gradient feedback (as defined in \eqref{eq:regret_full}) and two-point bandit feedback (as defined in \eqref{eq:regret_ban}) settings. We assume that all decision variables and the comparator point $x^*$ (in regret definition) lie in a geodesically convex subset $\mansubset$ of a Riemannian manifold $\man$. Furthermore, we assume that the local objective functions $f_{i,t}: \mansubset \to \mathbb{R}$ are geodesically convex. The studied problem is far from trivial due to the following technical challenges:

{\it (i) Decentralized Challenge:} In the finite-time analysis of all decentralized optimization algorithms, a fundamental step is to establish the linear variance reduction of a consensus step that averages local variables to align them towards the global objective. This property is well understood in the Euclidean space due to its linearity. However, direct extension of these results to curved Riemannian spaces is non-trivial due to their non-convexity. To mitigate this, many existing works either rely on the linearity of an ambient space by assuming an embedded submanifold \cite{chen2023local,deng2023decentralized,sarlette2009consensus},  or consider a specific scenario that all neighbors are equally weighted \cite{kraisler2023distributed,tron2012riemannian}. Also, the recent work of \cite{chen2024decentralized} introduces a fully decentralized and intrinsically defined consensus algorithm, establishing the linear variance reduction on Hadamard manifolds. Nonetheless, extending their analysis to more general (non-Hadamard) settings introduces network-dependent conditions for linear convergence, due to the added complexities of the positive curvature and the loss of global convexity. Consequently, achieving linear convergence for the consensus step on general manifolds remains a {\it significant and non-trivial open problem}. 

{\it (ii) Online Challenge:} The curved geometry of Riemannian manifolds introduces substantial complexities in online optimization. The first challenge concerns set-related operations: projections onto geodesically convex sets are no longer guaranteed to be nonexpansive, which necessitates a careful treatment of the resulting error terms in the regret analysis. The second challenge arises in the construction of gradient estimators for the bandit feedback setting. For example, the estimator proposed in \cite{wang2023online} assumes symmetric manifolds and requires computationally expensive calculations of surface area and volume. This motivates the need for practical, computationally efficient gradient estimators that avoid such overhead by smoothing techniques \cite{li2023stochastic,li2023zeroth}. However, their adoption in online setting requires {\it novel geodesic subconvexity analysis} of the smoothed objectives.

Our \textit{technical contributions} to address the  above challenges are as follows.
\begin{itemize}
\item We first establish the linear variance reduction property of the consensus step in \cite{chen2024decentralized} for non-Hadamard manifolds (Theorem \ref{thm:consensus}). Our analysis provides a unified path for decentralized optimization over both positively and negatively curved spaces. While \cite{chen2024decentralized} benefits from global negative curvature reinforcing convexity properties, our key {\it innovation} is to optimize the variance with respect to the consensus step-size to balance the trade-off between convexity and smoothness.

\item In decentralized online Riemannian optimization under the full gradient information setting, we establish a static regret bound of $O(\sqrt{T})$ for general manifolds (Theorem \ref{thm:staticregret}). In particular, we carefully handle the projection error and its effect on the variance reduction of the consensus step in positively curved spaces.  This result matches the optimal rate known in the Euclidean setting, indicating that neither decentralization nor the weakened convexity induced by positive curvature degrades the regret rate with respect to the time horizon $T$.
\item In the two-point bandit setting, we establish  {\it the first static regret bound} of $O(\sqrt{T})$ using a computationally efficient gradient estimator (Theorem~\ref{thm:DOR2BAN}). Our method uses the pullback of the function by exponential mapping with uniformly sampled directions in the tangent space \cite{li2023stochastic,li2023zeroth,sahinoglu2025finite}, yielding a smoothed objective that is subconvex with subconvexity linear in the smoothing parameter. We establish feasibility via the inclusion of balls on curved spaces through generalized comparison inequalities, and show that the errors from function approximation, domain shrinkage, and subconvexity are asymptotically negligible, preserving the overall $O(\sqrt{T})$ rate.
\item To support these results, we develop geometric tools for general Riemannian manifolds under a {\it bounded sectional curvature} assumption, where the sectional curvature lies in $[\kmin,\kmax]$, allowing both negative and positive bounds without requiring $\kmax \geq 0$ or $\kmin\leq 0$. Specifically, we adapt and extend existing results into {\it generalized comparison inequalities} valid for any sign of $\kmax$ and $\kmin$, with constants explicitly defined in terms of the curvature bounds (Lemmas \ref{lem:dlog-Pt}-\ref{lem:tri_distortion}). These tools underpin the analyses above and may be of independent interest for optimization over positively and negatively curved spaces.
\end{itemize}

\begin{table*}[t]
\caption{Static regret bounds for geodesically convex objectives; `*': separation oracle is used instead of projection, `**': linear optimization oracle is used instead of projection.}
\label{sampletable}
\vskip 0.15in
\begin{center}
\begin{small}
\begin{sc}
\begin{tabular}{lcccc}
\toprule
Reference & Manifold & Setting & Feedback & Regret Bound\\
\midrule

Wang et al.\cite{wang2023online} & Riemannian & Centralized & Gradient & $O(\sqrt{T})$\\

Wang et al.\cite{wang2023online} & Riemannian & Centralized & Two-Point Bandit & $O(\sqrt{T})$\\
Hu et al. \cite{hu2023riemannian} & Riemannian & Centralized & Gradient* & $O(\sqrt{T})$\\

Hu et al. \cite{hu2023riemannian} & Riemannian & Centralized & Gradient** & $O(T^{\frac{3}{4}})$\\

Chen and Sun \cite{chen2024decentralized} & Hadamard & Decentralized & Gradient & $O(\sqrt{T})$\\

\textbf{Our Work} & Riemannian & Decentralized & Gradient & $O(\sqrt{T})$\\

\textbf{Our Work} & Riemannian & Decentralized & Two-Point Bandit & $O(\sqrt{T})$\\

\bottomrule
\end{tabular}
\end{sc}
\end{small}
\end{center}
\vskip -0.1in
\end{table*}

\subsection{Literature Review}
\textbf{Decentralized Euclidean Optimization:} 
Decentralized optimization has been extensively studied in Euclidean spaces. For convex objectives, foundational algorithms such as distributed subgradient methods \cite{nedic2009distributed} and dual averaging \cite{duchi2011dual} established key convergence guarantees under static or slowly varying communication topologies. These early approaches have since been generalized to handle nonconvex and nonsmooth objectives \cite{lin2024decentralized,sahinoglu2024online}, giving rise to a wide array of algorithms based on subgradient methods \cite{chen2021distributed,yuan2016convergence}, gradient tracking \cite{nedic2017achieving,qu2017harnessing,shi2015extra}, and augmented Lagrangian or penalty-based frameworks \cite{aybat2017distributed,chang2014multi,xu2015augmented}. With the growing interest in optimization on manifolds, these decentralized techniques have increasingly been adapted to the Riemannian setting, where curvature and intrinsic geometry present new theoretical and algorithmic challenges.

\textbf{Decentralized Riemannian Optimization (DRO):} Early research in DRO primarily focused on establishing the asymptotic convergence of consensus algorithms \cite{sarlette2009consensus,tron2012riemannian}. More recently, attention has shifted to analyzing the non-asymptotic convergence  of such algorithms, particularly on specific manifolds (e.g., Stiefel manifold \cite{chen2023local,sun2024global}). Some works have extended the analysis to general compact submanifolds by employing an extrinsic projection approach grounded in the concept of proximal smoothness \cite{chen2023decentralized,deng2023decentralized,hu2023achieving}. Most recently, linear convergence of the decentralized consensus step has been established for Hadamard manifolds, marking a significant advancement in the theoretical understanding of DRO over nonpositively curved spaces \cite{chen2024decentralized}. However, a unifying framework that accommodates both positively and negatively curved manifolds remains an open challenge.

\textbf{Riemannian Online Optimization:} In the context of Riemannian optimization, the extension of online convex optimization (OCO) to manifold settings has recently garnered significant attention. Initial work focused on deriving regret bounds for geodesically convex objectives on Hadamard manifolds, demonstrating regret rates comparable to those achieved in the Euclidean OCO framework \cite{antonakopoulos2020online,hu2023riemannian,maass2022tracking,wang2023online}.  Subsequently, these results were extended to dynamic regret settings in \cite{hu2023minimizing,wang2023riemannian}. A recent line of work replaces geodesic convexity with the stronger notion of horospherical convexity and establishes curvature-independent regret bounds on Hadamard manifolds \cite{sahinoglu2025online}. However, these guarantees remain restricted to Hadamard manifolds and first order methods.  In \cite{wang2023online}, the study of online geodesically convex optimization was further broadened to include Riemannian bandit algorithms and extensions to manifolds with positive curvature. Most recently, the first decentralized regret bounds for Riemannian OCO were introduced in \cite{chen2024decentralized} for Hadamard manifolds, marking an important step toward distributed OCO in non-Euclidean settings. However, regret guarantees for more general Riemannian manifolds, particularly non-Hadamard manifolds, remains largely unexplored. Also, to the best of our knowledge, there is no prior work on decentralized Riemannian OCO in the bandit setting.

\section{Preliminaries}

In this section, we begin by introducing the geometric concepts fundamental to optimization over Riemannian manifolds. We then formally define the decentralized online optimization problem in the Riemannian setting. Finally, we present the technical assumptions that underpin our analysis.

\subsection{Background on Riemannian Optimization}

We consider a $d$-dimensional complete Riemannian manifold $\man$, equipped with a Riemannian metric $\metg$. For any point $x \in \man$, we denote its tangent space by $\tansp{x}$. The metric $\metg$ induces an inner product $\innerprod{\cdot}{\cdot}_x : \tansp{x} \times \tansp{x} \to \mathbb{R}$, which varies smoothly with $x$. The subscript $x$ is omitted hereafter for notational convenience. We denote by $\sptansp{x}{r}$ (respectively, $\balltansp{x}{r}$) the sphere (respectively,  ball) of radius $r$, centered at the origin of the tangent space $\tansp{x}$.

Geodesics on manifolds are generalizations of lines in Euclidean spaces, i.e., curves with constant speed that are locally distance-minimizing. Consequently, we can define the distance between two points on the manifold as the length of the geodesic $\gamma$, $d(x,y) := \inf_{\gamma} \int_0^1 \norm{\gamma'(t) }dt$, where $\gamma(0)=x$ and $\gamma(1)=y$. The exponential mapping on a Riemannian manifold, $\gamma(t)=\expm{x}(tv)$, defines a geodesic on the manifold, and the distance between $x$ and $\expm{x}(v)$ is $d(x,\expm{x}(v)) = \norm{v}$. 
Let us denote by $r_\mathrm{inj}$ the injectivity radius of the manifold. For any two points $x,y \in \man$ that satisfy $d(x,y)< r_\mathrm{inj}$, we define $\logm{x}(y):\man \to \tansp{x}$ as the inverse of exponential mapping. Also, $\ptg{x}{y}[u]$ denotes the parallel transport of $u\in \tansp{x}$ to some point in $\tansp{y}$ along the geodesic connecting $x$ to $y$. Another fundamental concept in Riemannian optimization is curvature, which characterizes the distortion introduced by the  intrinsic geometry of the manifold. Since our analysis is carried out intrinsically, we specifically consider the sectional curvature, which quantifies the curvature of two-dimensional sections of the manifold and plays a central role in convergence analysis.

Let us consider a smooth function $f:\man \to \mathbb{R}$ such that $\rgrad{f}(x) \in \tansp{x}$ denotes the Riemannian gradient of $f$ at $x$. For a geodesic $\gamma(t)$ with $\gamma(0)=x$,  $\ddt f(\gamma(t))|_{t=0} = \innerprod{\rgrad{f(x)}}{\gamma'(0)}$ \cite{absil2009optimization,boumal2023introduction}. A function $f:\man \to \mathbb{R}$ is said to be geodesically convex (g-convex) if we have $f(\gamma(t))\le (1-t)f(x)+t f(y)$ for any $x,y \in \man$, any geodesic $\gamma$ with $\gamma(0)=x$ and $\gamma(1)=y$, and for any $t \in [0,1]$. In terms of Riemannian gradient, a g-convex function $f$ satisfies $f(y)\geq f(x)+ \innerprod{\rgrad{f(x)}}{\logm{x}{y}}$. The notion of g-convexity can be relaxed to $\lambda$-g-subconvexity if the function satisfies $f(y)-f(x)-\innerprod{\rgrad{f(x)}}{\logm{x}{y}}\geq -\lambda$, for a constant $\lambda > 0$.

\subsection{Problem Formulation}
In decentralized online Riemannian optimization, $n$ 
agents collaborate to minimize a global objective function over a g-convex subset $\mansubset$ of the manifold $\man$. Each agent $i \in \{1,...,n\}$ observes a sequence of {\it local} loss functions $\{f_{i,t}:\mansubset \to \mathbb{R}\}_{t=1}^T$, and the {\it global} objective at time $t$ is given by the average function $f_t(x)=\frac{1}{n}\sumn{i}f_{i,t}(x)$. Each agent $i$ generates the decision variable $x_{i,t}$ only based on its previous history $\{f_{i,\tau}:\mansubset \to \mathbb{R}\}_{\tau=1}^{t-1}$ and the information received from its neighbors, defined based on a network model. 

{\bf Network Model:} The communication among agents is typically modeled by a doubly stochastic matrix 
\( W = [w_{ij}] \), where 
\( w_{ij} > 0 \) represents the weight assigned by agent 
\( i \) to the information received from agent 
\( j \) if agents $i$ and $j$ are connected; if the two agents are not connected, we have $w_{ij}=0$. To achieve consensus and drive agents towards the common goal (global objective), a standard step in decentralized Euclidean optimization algorithms is to use rows of $W$ for weighted averaging \cite{shahrampour2017distributed}. For example, $x_i=\sum_{j=1}^nw_{ij} y_j$ can be used to move variables $\{y_j\}_{j=1}^n$ towards their average. This heavily relies on the linear structure of Euclidean spaces, and with the absence of such linearity in Riemannian manifolds, a valid averaging scheme is the weighted Fr\'echet mean, i.e. $x_i=\mathrm{argmin}_{y\in \mansubset}\{ \sumn{j}w_{ij}d^2(y,y_j)\}$. Since this mechanism requires solving an optimization on a manifold itself, an alternative efficient averaging scheme can be used via 
\begin{equation}
\label{eq:consensus}
    x_i(s) = \expm{y_i} (s \sumn{j} w_{ij}\logm{y_i} y_j ),
\end{equation}
where $s$ is a control parameter that depends on the curvature of the manifold \cite{chen2024decentralized}. We use this update in our algorithm design and demonstrate its linear variance reduction property in Section \ref{sec:consensus}.

{\bf Regret Definition:} Under the full information feedback setting, a common goal is to minimize static regret with respect to a fixed comparator $x^* \in \mansubset$, where $\mansubset$ is a geodesically convex subset of $\man$, and
\begin{equation}
\label{eq:regret_full}
    \regstafull := \regstafulleq.
\end{equation}
This definition is standard in decentralized online optimization \cite{chen2024decentralized, shahrampour2017distributed}. Note that while in a decentralized algorithm, agent $i$ generates ${x_{i,t}}$ using Riemannian gradients of $\{f_{i,\tau}:\mansubset \to \mathbb{R}\}_{\tau=1}^{t-1}$, its decision is evaluated at the global function $f_t$, so without communication the regret will not be sublinear. 

In the two-point bandit setting, where only function  evaluations at two nearby points are available, the regret is similarly defined by
\begin{equation}
\label{eq:regret_ban}
    \regstaban:=\regstabaneq,
\end{equation}
where $x_{i,t,1}$ and $x_{i,t,2}$ are used to estimate the Riemannian gradient at $x_{i,t}$ (see Algorithm \ref{alg:DOR_ban}). This definition extends the classical two-point bandit formulation in \cite{agarwal2010optimal} to Riemannian manifolds and aligns with prior adaptations based on alternative sampling strategies in \cite{wang2023online}.

\subsection{Technical Assumptions}

We assume that agents communicate synchronously through the network. The communication matrix $W$ is fixed over time and satisfies the following assumption.

\begin{assumption}
\label{assum:com}
The network is connected and the communication matrix $W \in \mathbb{R}^{n \times n}$ is symmetric and doubly stochastic. $\sigma_2(W)$ denotes the second largest singular value of the matrix $W$, for which we have that $\sigma_2(W) \in [0,1)$.  

Assumption \ref{assum:com} is widely used in the decentralized optimization literature (see e.g., \cite{duchi2011dual,shi2015extra,shahrampour2017distributed}). Regardless of whether the network structure is fixed or time-varying, some form of connectivity assumption (e.g., bounded intercommunication intervals \cite{nedic2009distributed}) is required to ensure convergence. In this context, the quantity $\sigma_2(W)$ characterizes the connectivity of the network: a smaller $\sigma_2(W)$ corresponds to a better-connected network, facilitating faster information propagation and consensus among the agents.
\end{assumption}

\begin{assumption}
\label{assum:cur}
We assume that 
\begin{enumerate}[label=(\roman*)]
\item The sectional curvature $K$ on $\man$ is bounded from below and above such that $\kmin\le K \le \kmax$.
\item The diameter of set $\mansubset$ is bounded by $D$. If $\kmax>0$, we further assume that $D < \rcx$, where $\rcx:=\frac{1}{2}\min\{\rinj,\frac{\pi}{\sqrt{\kmax}}\}$. 
\end{enumerate}
\end{assumption}
The assumption that the domain is not infinitely curved is standard in Riemannian optimization. In particular, nonpositive values of $\kmax$ correspond to Hadamard manifolds, which are widely studied due to their favorable geometric properties \cite{bacak2014convex,zhang2016first,de2019computing,sakai2023convergence}. Although positive curvature can adversely impact the convexity properties of objective functions, we allow $\kmax$ to be positive to include positively curved manifolds, thereby broadening the applicability of our results. The second part of the assumption ensures that the domain is {\it uniquely} geodesically convex, which is crucial for guaranteeing the well-posedness of optimization problems on the manifold \cite{wang2023online}.

The injectivity radius $\rinj$ and convexity radius $\rcx$ play a fundamental role in Riemannian optimization since many geometric operations are only well defined locally. The injectivity radius guarantees the local invertibility of the exponential and logarithm maps, ensuring that quantities such as $\logm{x}(y)$ and minimizing geodesics are uniquely defined. The convexity radius further guarantees geodesic convexity of neighborhoods, so that geodesics between nearby points remain inside the region of interest. These radius bounds are essential for establishing many geometric estimates used in convergence analysis, including smoothness and strong convexity inequalities, Hessian bounds for squared distance functions, stability of Fréchet means, and curvature dependent bounds on parallel transport and the differential of the exponential map. Consequently, many Riemannian optimization analyses require the iterates to remain within balls whose radius is controlled by the injectivity or convexity radius.

 \begin{assumption}\label{assum:fun}
     We assume that for all $i\in \{1,...,n\}$ local objectives $\{f_{i,t}\}_{t=1}^T$ are g-convex and $L$-Lipschitz on the domain $\mansubset$, which is a g-convex subset of the manifold $\man$.
 \end{assumption}

G-convexity of the objective functions as well as g-convexity of $\mansubset$ are standard assumptions in Riemannian OCO for deriving sublinear regret bounds \cite{hu2023achieving,wang2023online} as well as in the convergence rate analysis of first-order methods \cite{ahn2020nesterov,zhang2016first}. 

We now present the following properties, which are instrumental in the subsequent technical analysis.

\begin{lemma}[Corollary 2.1 of \cite{alimisis2020continuous}, Lemma 5 of \cite{zhang2016first}]
\label{lem:riem_cos}

Let $a,b,c \in \man$ with $d(a,b)\leq D$ and $d(b,c)\leq D$, so that log maps are well-defined. Then, under Assumption \ref{assum:cur} we have
\begin{align}
\label{eq:cosbounds}
    d^2(a,c) &\le g_1(\kmin, d(a,b)) d^2(b,c) + d^2(a,b) \nonumber \\ 
    &- 2\innerprod{\logm{b}(a)}{\logm{b}(c)}, \nonumber\\
    d^2(a,c) &\geq g_2(\kmax,q)d^2(b,c) + d^2(a,b) \nonumber \\ 
    &-2\innerprod{\logm{b}(a)}{\logm{b}(c)},
\end{align}
for some $q>0$, where $g_1(\cdot,\cdot)$ and $g_2(\cdot,\cdot)$ are defined in the Appendix \ref{app:constant}.

\end{lemma}
A major challenge in analyzing the non-asymptotic convergence of first-order methods in geodesic spaces is the absence of  Euclidean cosine law. In general nonlinear spaces, there are no direct analytical analogs. Therefore, in our analysis, we rely on the set of geometric inequalities \eqref{eq:cosbounds} to compare the edge lengths of geodesic triangles and to relate them to the inner products of tangent vectors.

\begin{lemma}[Lemma 4 of \cite{sun2019escaping}]
\label{lem:dist_lower}
Let $x,y,z \in \man$ with the distance of each two points being no larger than $D$. Then, under Assumption \ref{assum:cur} we have
\begin{align*}
(1+C_3D^2)^{-1} d(y,z) &\leq \norm{\logm{x}(y) - \logm{x}(z)}\nonumber\\ 
&\leq (1+C_4 D^2) d(y,z).
\end{align*}
\end{lemma}
This lemma establishes a relationship between the distance of points $y$ and $z$ on the manifold and the distance between their preimages under the exponential map centered at another point $x \in \man$. 
\section{Linear Variance Reduction of the Consensus Step }\label{sec:consensus}
A major component of all decentralized algorithms is a consensus step to drive agents towards a common goal. In decentralized Riemannian methods, the existing results mainly focus on Hadamard manifolds \cite{chen2024decentralized}, where the favorable curvature structure (i.e., $\kmax=0$) ensures the global g-convexity of the squared distance function. Alternatively, some approaches assume a submanifold structure \cite{deng2023decentralized,hu2023achieving} and perform consensus in the ambient Euclidean space, thus circumventing intrinsic geometric challenges.

In this work, we go beyond these settings and establish the linear convergence of the consensus step \eqref{eq:consensus} on Riemannian manifolds that are potentially positively  curved. The tangent-space consensus update of the form \eqref{eq:consensus} originates from the consensus on manifolds literature, with early formulations on specific manifolds in \cite{sarlette2009consensus}. This approach was later extended to general Riemannian manifolds in \cite{tron2012riemannian}, where asymptotic convergence was established under bounded sectional curvature. More recently, \cite{chen2024decentralized} proved linear variance reduction for Hadamard manifolds for a given communication matrix $W$. 

In the analysis of the consensus update, the first step is to bound the consensus variance in terms of pairwise geodesic distances between local variables. Using this bound, we later define the linear convergence coefficient and identify the optimal step size for the intrinsic consensus iteration, without relying on extrinsic approximations.

\begin{lemma}
\label{lem:dec_ineq}
Let Assumptions \ref{assum:com} and \ref{assum:cur} hold. Consider $n$ points $\{y_1,...,y_n\}$ on the subset $\mansubset$ of manifold $\man$, and let $\bar{y}$ be the Fr\'echet mean of these points. Then, we have 
    \begin{align}\label{eq:dec_ineq}
        \varset{y_i}&:=\frac{1}{n}\sumn{i}d^2(y_i,\bar{y}) \nonumber \\ &\le \frac{1}{n} \frac{(1+C_4D^2)^2}{2(1-\sigma_2(W))} \sumn{i}\sumn{j} w_{ij} d^2(y_i,y_j).
    \end{align}
\end{lemma}
This result establishes a link between the global consensus objective and the sum of local consensus objectives maintained by individual agents. Specifically, for agent $i$, define the local consensus objective as $g_i(y)=\frac{1}{2}\sumn{j}w_{ij}d^2(y,y_j)$. Then, the right-hand side (RHS) of \eqref{eq:dec_ineq} is proportional to a global consensus objective $\sum_{i=1}^ng_i(y_i)$. Each agent $i$ updates its variable by moving in the direction that minimizes its local objective $g_i(y)$, since the Riemannian gradient evaluated at $y_i$ is $\rgrad{g_i(y_{i})} = -\sumn{j}w_{ij}\logm{y_i}(y_j)$, and basically, the update \eqref{eq:consensus} can be written as $\expm{y_i}(-s~\rgrad{g_i(y_{i})})$. Also, minimizing the RHS of \eqref{eq:dec_ineq} collectively acts as a mechanism to reduce the variance (that is, LHS of \eqref{eq:dec_ineq}). Building on this result, we present the following theorem, which establishes the linear convergence of the consensus step under a fixed consensus step size $s$.

\begin{theorem}
\label{thm:consensus}
Let Assumptions \ref{assum:com} and \ref{assum:cur} hold, consider the consensus step \eqref{eq:consensus} for $n$ points $\{y_1,...,y_n\}$ on $\mansubset\subseteq\man$, and let $\bar{y}$ be the Fr\'echet mean of these points. Selecting the step size $s=(2C_1)^{-1}C_2$, we achieve a linear variance reduction with the rate parameter $\rho\in (0,1)$, where $\rho:=1 -\frac{C_2^2(1-\sigma_2(W))}{2C_1(1+C_4D^2)^2}$ and
\begin{align}
    \varset{x_i(s)}\le \frac{1}{n}\sumn{i} d^2(x_i(s),\bar{y}) &\leq  \frac{\rho}{n} \sumn{i} d^2(\bar{y},y_i) \nonumber\\ 
    &= \rho \varset{y_i}.
\end{align}
Here, $\sigma_2(W)$ denotes the second largest singular value of the matrix $W$, and the curvature related constants are defined in Appendix \ref{app:constant}.
\end{theorem}
The primary challenge in deriving this result is balancing the effect of curvature. Negative curvature weakens smoothness, while positive curvature weakens convexity. Our {\it key innovation} is to optimize the variance with respect to the consensus step size $s$ to balance this trade-off. Unlike existing methods tailored for Hadamard manifolds \cite{chen2024decentralized}, which benefit from global negative curvature and favorable convexity properties, we adopt a more geometric approach. Specifically, we leverage Lemma \ref{lem:dist_lower} to upper bound the distortion introduced by positive curvature through a multiplicative factor, enabling a unified analysis beyond the Hadamard setting. 

\begin{remark}[Effect of $\kmax$]
If the manifold $\man$ is a Hadamard manifold, the maximum sectional curvature $\kmax=0$. In this case, the only factor contributing to the slowdown of the algorithm is the potentially large upper bound on the Hessian of the squared distance function induced by $\kmin$.
\end{remark}

\begin{remark}
Theorem \ref{thm:consensus} serves as a fundamental building block for DRO on manifolds with bounded sectional curvature, since it can be used to analyze the distance between local variables and their (global) average. While we use it for the network error analysis, the result can be of separate interest. 
\end{remark}

\begin{remark}
The consensus step that generates
$x_{i,t+1}$ does not include an additional projection onto $\mansubset$. We assume invariance condition that the iterates of consensus step remain in $\mansubset$. This type of invariance assumption is standard in the literature; see, e.g.,\cite{ahn2020nesterov,alimisis2021momentum}.
\end{remark}

\section{Decentralized Online Riemannian Optimization: Full Information}\label{sec:full}

\begin{algorithm}[tb]
    \caption{Decentralized Online Riemannian Gradient Descent Algorithm }
   \label{alg:DOR_full}
\begin{algorithmic}

    \STATE {\bfseries Input:} $\mansubset \subseteq \man$, gradient step-size $\eta$, consensus step-size $s$, initial point $x_{i,1}=x_1 \in \mansubset$ 
    \FOR{$t=1$ {\bfseries to} $T$}

    \STATE$g_{i,t} = \rgrad{f_{i,t}}(x_{i,t})$
    \STATE $y_{i,t+1} = \projman{\mansubset}(\expm{x_{i,t}}(-\eta g_{i,t}))$

    \STATE $x_{i,t+1}=\expm{y_{i,t+1}}( s \sum_{j=1}^{n} w_{ij} \logm{y_{i,t+1}}(y_{j,t+1}) )$
    
    \ENDFOR

\end{algorithmic}
\end{algorithm}

In this section, we present the decentralized online Riemannian optimization algorithm (Algorithm~\ref{alg:DOR_full}) \cite{chen2024decentralized} under full-information feedback and establish an upper bound of $O(\sqrt{T})$ for the regret \eqref{eq:regret_full} for manifolds that are potentially positively curved. In Algorithm~\ref{alg:DOR_full}, the optimization proceeds over a time horizon of $T$ iterations. At iteration $t$, each agent $i$ receives a local Riemannian gradient $g_{i,t}\in \tansp{x_{i,t}}$, applies exponential mapping to bring the vector back to the manifold $\man$, and then applies a Riemannian projection mapping $\projman{\mansubset}(x):=\text{arg min}_{y \in \mansubset} d(x,y)$ to ensure feasibility. The projection oracle always returns a unique solution for small enough gradient step-size $\eta$. The agents then perform a consensus step following \eqref{eq:consensus} to move towards the global objective $f_t$.

The static regret analysis of Algorithm \ref{alg:DOR_full} relies on the decomposition of the regret into two main components. The first component, known as the network error, captures the discrepancy between the local objectives and the global objectives, and it can be bounded using Theorem~\ref{thm:consensus}.

To ensure a good approximation of the global objective, it is desirable to maintain a small geodesic distance $d(x_{i,t},\bar{x}_t)$ between each agent's variable and the network Fr\'echet mean. While the consensus step drives a reduction in the variance among local variables, the gradient updates that follow can reintroduce divergence. As a result, it is essential to establish an upper bound on $d(x_{i,t},\bar{x}_t)$ that incorporates both the linear convergence properties of the consensus step and the influence of the learning rate in the local gradient updates.

\begin{lemma}[Network Error]
\label{lem:net_error}
Let Assumptions \ref{assum:com}, \ref{assum:cur} and \ref{assum:fun} hold. Running Algorithm \ref{alg:DOR_full} on the local variables $x_{i,t}$ with $s=(2C_1)^{-1}C_2$ results in a bounded network error as
\begin{equation}
    d(x_{i,t},\bar{x}_t) \leq \frac{2\sqrt{n}\eta L}{1-\rho}, 
\end{equation}
where $\rho$ is defined in Theorem \ref{thm:consensus}.
\end{lemma}
Lemma \ref{lem:net_error} establishes that the network error exhibits a $O(\eta)$ dependence on the gradient step-size $\eta$, which will later be optimized as a function of the time horizon $T$ when analyzing the regret bound.

The second term in the static regret decomposition involves the expression $\sumt f_{i,t}(x_{i,t})-f_t(x^*)$, which can be bounded by leveraging the geodesic convexity of the local objective functions. A key distinction from the Euclidean setting lies in the lack of nonexpansiveness in the projection step on curved manifolds. This geometric complication introduces additional error terms into the regret analysis. By deriving upper bounds for both the network error and the local optimization error, we establish the following static regret bound for Algorithm~\ref{alg:DOR_full}.

\begin{theorem}[Static Regret-Full Information]
\label{thm:staticregret}
Suppose that Assumptions \ref{assum:com}, \ref{assum:cur} and \ref{assum:fun} hold. Running Algorithm \ref{alg:DOR_full} for $T$ iterations with $\eta=O(1/\sqrt{T})$ and $s=(2C_1)^{-1}C_2$ gives the following  static regret bound
\begin{equation}
     \regstafull = \frac{1}{n} \sum_{t=1}^T \sumn{i} f_t(x_{i,t}) - \underset{x\in \mansubset}{\min} \sum_{t=1}^T f_t(x) \leq DC_5 \sqrt{T},
\end{equation}
where $C_5$, defined in Appendix \ref{app:constant}, is independent of $T$. 
\end{theorem}

The static regret of decentralized online Riemannian optimization matches the same regret bound achieved by both its centralized Riemannian counterparts \cite{hu2023riemannian , wang2023online}  and decentralized Euclidean counterparts \cite{hosseini2016online,shahrampour2017distributed,yan2012distributed}, demonstrating that curvature and decentralized communication do not introduce additional asymptotic penalties in the regret bound under our assumptions. This result also certifies that $O(\sqrt{T})$ static regret can be achieved for decentralized online Riemannian optimization beyond Hadamard manifolds.

\section{Decentralized Online Riemannian Optimization: Bandit Feedback}\label{sec:bandit}
We now focus on the two-point bandit feedback setting, where gradient information is unavailable, and we need to construct a suitable gradient estimator using the pullback function \cite{li2023stochastic,li2023zeroth}. This is also known as {\it randomized smoothing}, and here, our main challenge is to establish the g-subconvexity of these estimators (under our assumptions) and to rigorously quantify the regret in the presence of additional errors introduced by function approximation, domain shrinkage, and subconvexity.

In Algorithm~\ref{alg:DOR_ban}, the key difference from the full-information setting is the use of gradient estimators $\gdel_{i,t}$ instead of exact Riemannian gradients. Each $\gdel_{i,t}$ is constructed by sampling a direction $u_{i,t}$ from a unit sphere in the tangent space $\tansp{x_{i,t}}$. Since the points $x_{i,t,1}$ and $x_{i,t,2}$ lie within a $\delta$-ball around $x_{i,t}$, we ensure feasibility by restricting the iterates to a shrinking subset $(1-\tau)\mansubset$ of the original domain $\mansubset$, where the shrinkage factor $\tau$ depends on $\delta$. The shrinking set can be defined as $\{\expm{p}((1-\tau)\logm{p}y)|y\in \mansubset\}$ with respect to an interior point $p$ of $\mansubset$. Consequently, the original regret \eqref{eq:regret_ban}, defined with respect to the functions $f_{i,t}$ over the set $\mansubset$, is transformed into a regret bound involving the smoothed functions $\fdel_{i,t}$ over a smaller domain $(1-\tau)\mansubset$, where $\fdel_{i,t}(x):=\int f_{i,t}(\expm{x}(u)) d\mu(u) $ is the smoothed version of $f_{i,t}(x)$, $\mu$ is a uniform measure on $\balltansp{x}{\delta}$, and $\gdel_{i,t}$ approximates the gradient of $\fdel_{i,t}$.

\begin{algorithm}[tb]
    \caption{Decentralized Online Riemannian Two-Point Bandit Algorithm }
   \label{alg:DOR_ban}
\begin{algorithmic}

    \STATE {\bfseries Input:} $\mansubset \subseteq \man$ and intrinsic dimension $d$, gradient step-size $\eta$, consensus step-size $s$, shrinking domain $(1-\tau)\mansubset$ with shrinkage factor $\tau$, smoothing parameter $\delta$, initial point $x_{i,1}=x_1 \in (1-\tau)\mansubset$
    \FOR{$t=1$ {\bfseries to} $T$}

    \STATE Sample $u_{i,t}$ uniformly from $\sptansp{x_{i,t}}{1}$

    \STATE Let $x_{i,t,1}= \expm{x_{i,t}}(\delta u_{i,t})$  and $x_{i,t,2}= \expm{x_{i,t}}(-\delta u_{i,t}) $

    \STATE$\gdel_{i,t} = \frac{d}{2\delta} (f_{i,t}(x_{i,t,1}) - f_{i,t}(x_{i,t,2}))u_{i,t}$
    \STATE $y_{i,t+1} = \projman{(1-\tau)\mansubset}\big(\projman{\mansubset}(\expm{x_{i,t}}(-\eta \gdel_{i,t}))\big)$ 
    
    \STATE $x_{i,t+1}=\expm{y_{i,t+1}}( s \sum_{j=1}^{n} w_{ij} \logm{y_{i,t+1}}(y_{j,t+1}) )$
    
    \ENDFOR
    
\end{algorithmic}
\end{algorithm}

Working with smoothed objectives introduces several technical challenges. The first one is the cost of approximation with smoothed objectives, which can be controlled with the smoothing parameter $\delta$ and the Lipschitz constant $L$.  The second issue is the projection error arising from restricting updates to the feasible shrinking set $(1 - \tau)\mansubset$. 
The third and more subtle challenge concerns the subconvexity properties of the smoothed objective $\fdel_{i,t}$. Although $f_{i,t}$ are geodesically convex (Assumption \ref{assum:fun}), the smoothing operation does not preserve g-convexity and makes the gradient of the smoothed objective difficult to handle directly. While for the first two challenges we can adapt the techniques in prior work, the third remains an open problem and requires new analysis specific to Riemannian geometry. Instead of relating the gradient of the smoothed objective to $\gdel_{i,t}$, we directly establish the subconvexity inequality over $\gdel_{i,t}$ along with its approximation error in terms of $\delta$, thereby quantifying the impact of curvature and smoothing on geodesic subconvexity.

In the following lemma, we formally establish that local functions $\fdel_{i,t}$ are $O(\delta)$-g-subconvex, so their deviation from perfect geodesic convexity is controlled linearly by the smoothing parameter $\delta$.

\begin{lemma}
\label{lem:subconvexity}
Suppose that Assumption \ref{assum:cur} holds, and let $f$ be any g-convex and $L$-Lipschitz function on the g-convex domain $\mansubset \subseteq \man$. Let $\fdel(x)=\int f(\expm{x}(u)) d\mu(u)$ be the smoothed objective with $\mu$ denoting a uniform measure on $\balltansp{x}{\delta}$. Define the gradient estimator  
\begin{equation}
    \gdel (x) = \frac{d}{2\delta} (f(\expm{x}(\delta u)) - f(\expm{x}(-\delta u)))u,
\end{equation}
where $u$ is uniformly distributed on $\sptansp{x}{1}$. Then, the expected estimator $\mathbb{E}_u [g^{\delta}(x)] \in \tansp{x}$ certifies $\delta L C_6$ g-subconvexity of $f^{\delta}$ in the following sense
\begin{equation}
    f^{\delta}(y)-f^{\delta}(x) -\innerprod{\mathbb{E}_u[g^{\delta}(x)]}{\logm{x}{(y)}} \geq -\delta LC_6,
\end{equation}
for all $x,y\in (1-\tau)\mansubset$, where $C_6>0$ depends on $\kmax$, $\kmin$, and $D$ (see Appendix \ref{app:constant}).
\end{lemma}

\begin{figure*}[t!]
\begin{center}
\centerline{\includegraphics[width=1.8\columnwidth,height=0.3\textwidth]{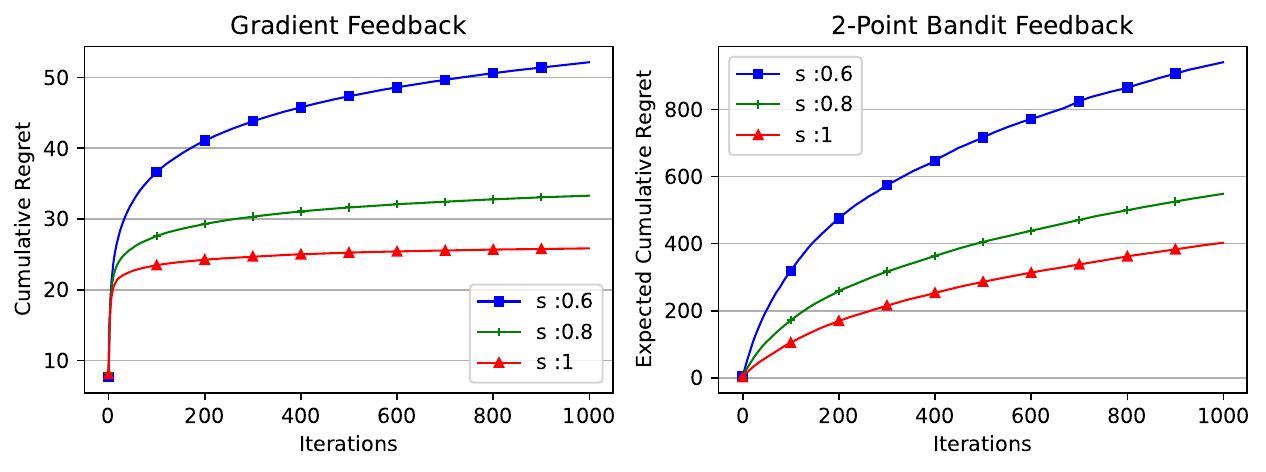}}
\caption{Cumulative regret with gradient step-size $\eta=\frac{1}{\sqrt{t}}$ and consensus step-size $s\in \{0.6,0.8,1\}$.}
\label{fig:full}
\end{center}
\vskip -0.18in
\end{figure*}

In the two-point bandit setting, the additional terms introduced by the smoothing operation are controlled by the smoothing parameter $\delta$. By selecting a sufficiently small $\delta$, the approximation error and the loss of convexity due to smoothing can be made negligible. Specifically, for the static regret derivation, it suffices to choose $\delta =  O(1/T)$, ensuring that the extra cost introduced by smoothing does not affect the overall regret rate. Building on these observations, we present the following theorem, which establishes the {\it first} $O(\sqrt{T})$ static regret bound for decentralized online Riemannian optimization under the two-point bandit feedback setting.

\begin{theorem}[Static Regret-Bandit Feedback]
\label{thm:DOR2BAN} 
Suppose that assumptions \ref{assum:com}, \ref{assum:cur}, and \ref{assum:fun} hold. Let $x_{i,t,1}$ and $x_{i,t,2}$ be points generated by Algorithm \ref{alg:DOR_ban}. If we take $\delta=O(1/T)$ and $\tau = O(\delta)$, the expected regret of Algorithm \ref{alg:DOR_ban} is bounded by 
\begin{equation}
    \mathbb{E}[\regstaban] \leq O(\eta^{-1}+\eta T+\sqrt{T}). 
\end{equation}
Therefore, the choice of $\eta = O(1/\sqrt{T})$ results in $O(\sqrt{T})$ regret bound. 
\end{theorem}

Although curved spaces introduce additional challenges, such as projection errors and loss of subconvexity, beyond those encountered in the Euclidean setting (e.g., function approximation due to smoothing), the same static regret bound of $O(\sqrt{T})$ can still be achieved in the two-point bandit setting. Our result also recovers the regret rate of \cite{wang2023online} for centralized online Riemannian optimization with two-point bandit feedback.

\section{Numerical Experiments}
In this section, we conduct experiments to evaluate the performance of our algorithms. We focus specifically on a positively curved manifold, the unit sphere in $16$-dimensional Euclidean space, which poses greater geometric challenges compared to Hadamard manifolds due to its limited injectivity radius and the potential for projection errors. We consider $n=50$ agents connected via a ring graph topology, where each agent communicates with its $10$ immediate neighbors. 

{\bf Objective Function:} The task is to compute the decentralized Fr\'echet mean in an online manner, where the local objective is defined as $f_{i,t}(x)=d^2(x,z_{i,t})$ for a given point $z_{i,t}$, and the global objective is $f_t(x)=\frac{1}{n}\sumn{i} d^2(x,z_{i,t})$. We define $\mansubset$ as a geodesic ball with radius $\frac{\pi}{4}$. To generate $z_{i,t}$, we sample base points $\{z_i\}_{i=1}^n$ uniformly on $\mansubset$. Then, in each iteration $t$, agent $i$ receives information from the local function by sampling $z_{i,t}$ uniformly from a $\frac{\pi}{16}$-ball centered at $z_i$.

{\bf Hyperparameters:} In the first experiment, we focus on the full information feedback (Algorithm~\ref{alg:DOR_full}). We run the algorithm using an adaptive step-size $\eta = 1/\sqrt{t}$ to observe cumulative regret with respect to different timescales. We choose consensus step-sizes as $s \in \{0.6,0.8,1\}$. We also evaluate the bandit setting by running Algorithm~\ref{alg:DOR_ban} with the same step-sizes. Also for the bandit setting, since the influence of the smoothing and shrinkage parameters $\delta$ and $\tau$ on the overall regret is negligible, we choose them sufficiently small as $\delta = \tau = \pi/50$.

{\bf Performance:} The resulting cumulative static regrets are shown in Fig.~\ref{fig:full}. To mitigate the variance in the bandit feedback setting, we report the average cumulative regret over $8$ Monte Carlo simulations. Across both experiments, we observe that larger consensus step-size $s$ leads to a smaller regret, primarily due to improved consensus rate. While our theoretical analysis requires $s<1$ to ensure stability in worst-case scenarios, we ran the simulations even for $s=1$ to show that in practice, convergence may still be achieved for larger step-sizes. Furthermore, in the bandit setting, the convergence is noticeably slower (i.e., larger regret), which aligns with theoretical expectations, because the variance in gradient estimation leads to a larger regret, especially during early iterations. This effect is further exacerbated by the dimensional dependence of the gradient estimator, which contributes to a slower convergence rate compared to the full-information setting.

\section{Conclusion}
We addressed decentralized online optimization over manifolds with possibly positive curvature. (i) We established the  linear variance reduction proprty for the consensus step \eqref{eq:consensus}, (ii) proved a $O(\sqrt{T})$ regret bound for the gradient feedback setting, and (iii) demonstrated the same $O(\sqrt{T})$ regret bound for the bandit setup through a subconvexity analysis of smoothed objectives. Based on the existing lower bounds in the Euclidean setting, our regret bounds are optimal in terms of the time horizon $T$. Another strength of our results is the fact that they are derived under mild and standard technical assumptions, commonly used in the Riemannian optimization literature. However, it is of separate interest to analyze the dependence of these bounds to other parameters, such as network connectivity and manifold curvature. Establishing lower bounds that characterize the fundamental dependence on curvature and network properties, and designing algorithms that are provably optimal with respect to these parameters, are interesting directions for future research.

\section*{Appendix}
In this section, we present proofs of theorems and lemmas in Sections \ref{sec:consensus}, \ref{sec:full} and \ref{sec:bandit}. 
\subsection{Proof of Lemma \ref{lem:dec_ineq}}
\begin{proof}
Recall that $\bar{y}$ is the Fr\'echet mean of points $\{y_i\}_{i=1}^n$. We start with using Lemma \ref{lem:dist_lower} for points $\bar{y},y_i,y_j\in \mansubset$, where their pairwise distances are no larger than $D$. We have that
\begin{equation*}
    d^2(y_i,y_j) \geq \frac{\norm{\logm{\bar{y}}(y_i) - \logm{\bar{y}}(y_j)}^2}{(1+C_4 D^2)^2}.
\end{equation*}
Due to doubly stochasticity of $W$, the above implies
\resizebox{\columnwidth}{!}{%
\parbox{\columnwidth}{%
\begin{align*}
&\sumn{i} \sumn{j} w_{ij} d^2(y_i,y_j) \nonumber \\ 
&\geq \frac{2\sumn{i}\norm{\logm{\bar{y}}(y_i)}^2 - 2\sum_{i,j=1}^n w_{ij}\innerprod{\logm{\bar{y}}(y_i)}{\logm{\bar{y}}(y_j)}}{(1+C_4 D^2)^2} \nonumber \\ 
&= \frac{2\sumn{i} d^2(\bar{y},y_i) - 2\sum_{i,j=1}^{n} w_{ij}\innerprod{\logm{\bar{y}}(y_i)}{\logm{\bar{y}}(y_j)}}{(1+C_4 D^2)^2}.
\end{align*}
}%
}
By definition, $\bar{y}:=\mathrm{argmin}_{y\in \mansubset}\{\sumn{i} d^2(y,y_i)\}$, so the Riemannian derivative of the objective satisfies the stationarity condition $\sumn{i} \logm{\bar{y}}(y_i) = 0$ at $\bar{y}$. 

Let us define $W' \in \mathbb{R}^{n \times n}$ so that $W' := W - \frac{1}{n} 1_n 1_n^\top$, where $1_n$ is the vector of all ones (with dimension $n$). Then, $\sigma_1(W')$, the largest singular value of $W'$, is equal to $\sigma_2(W)$.

Since $\logm{\bar{y}}(y_i) \in \tansp{\bar{y}}$, we can construct a matrix $V$ by stacking $\{\logm{\bar{y}}(y_i)\}_i$ in the columns of $V$. Then, 
\begin{align*}    \vphantom{\sumn{i}}\sumn{i}\sumn{j}&w_{ij}\innerprod{\logm{\bar{y}}(y_i)}{\logm{\bar{y}}(y_j)}\\ 
&= \sumn{i}\sumn{j} (w_{ij}-\frac{1}{n})\innerprod{\logm{\bar{y}}(y_i)}{\logm{\bar{y}}(y_j)}\\
    \vphantom{\sumn{i}}&= \text{Tr}(V W'V^\top)\\
    \vphantom{\sumn{i}}&\le \sigma_1(W') \text{Tr}(V^\top V)\\
    \vphantom{\sumn{i}}&= \sigma_2(W) \sumn{i} d^2(\bar{y},y_i).
\end{align*}
Combining these results, we obtain
\begin{equation*}
   \sumn{i} \sumn{j} w_{ij} d^2(y_i,y_j) \geq \frac{2(1-\sigma_2(W))}{(1+C_4 D^2)^2} \sumn{i} d^2(\bar{y},y_i).
\end{equation*}
\end{proof}

\subsection{Proof of Theorem \ref{thm:consensus}}
\begin{proof}
\label{prf:consensus}
We start with upper bounding $d^2(x_i(s),\bar{y})$ by applying Lemma~\ref{lem:riem_cos} on $x_i(s),\bar{y},y_i\in \mansubset$ and using the consensus update \eqref{eq:consensus} as
\begin{align*}
    d^2(x_i(s),\bar{y}) &\leq d^2(y_i,\bar{y}) + g_1(\kmin, d(y_i,\bar{y})) d^2(y_i,x_i(s))\\
    &- 2s\innerprod{\logm{y_i}(\bar{y})}{\sumn{j}w_{ij} \logm{y_i}{(y_j)}}. 
\end{align*}
We can write $d^2(y_i,x_i(s))=\norm{\logm{y_i}{(x_i(s))}}^2=s^2 \norm{\sumn{j}w_{ij} \logm{y_i}{(y_j)}}^2 $. Let us define $C_1:= g_1(\kmin,D)$. Then, $g_1(\kmin, d(y_i,\bar{y})) \le C_1$ since $g_1$ is an increasing function in its second argument when it is positive. Hence, we can write 
\begin{align*}
    d^2(x_i(s),\bar{y}) &\leq d^2(y_i,\bar{y}) - 2s \sumn{j}w_{ij}\innerprod{\logm{y_i}(\bar{y})}{\logm{y_i}{(y_j)} }\\
    &+ s^2 C_1 \norm{\sumn{j}w_{ij} \logm{y_i}{(y_j)}}^2.
\end{align*}
Summing above over $i$ gives the following inequality 
\begin{align*}
\sumn{i} d^2(x_i(s),\bar{y}) &\leq \sumn{i} d^2(y_i,\bar{y})\\  
&-  2s\underbrace{\sumn{i} \sumn{j} w_{ij}\innerprod{\logm{y_i}(\bar{y})}{ \logm{y_i}{(y_j)} } }_{T_1}\\ 
&+s^2\underbrace{C_1\sumn{i}\norm{\sumn{j}w_{ij} \logm{y_i}{(y_j)}}^2}_{T_2}.    
\end{align*}
We now need to find a lower bound on $T_1$ in terms of $\sumn{i}d^2(y_i,\bar{y})$. By applying Lemma \ref{lem:riem_cos} to points $y_i,\bar{y},y_j\in\mansubset$, we can write
\begin{align*}
\innerprod{\logm{y_i}(\bar{y})}{\logm{y_i}{(y_j)}} &\geq \frac{1}{2} \Big(d^2(y_i,\bar{y}) - d^2(y_j,\bar{y}) \\ &+ g_2(\kmax,d(\bar{y},y_i)) d^2(y_i,y_j) \Big).
\end{align*}
Let $C_2:=g_2(\kmax,D)$. Since $g_2$ is decreasing in its second argument on $[0,\pi/2]$, $C_2$ is a lower bound on $g_2(\kmax,d(\bar{y},y_i))$. We can use this inequality to lower bound $T_1$ as
\begin{align*}
    T_1 &= \sumn{i} \sumn{j} w_{ij}\innerprod{\logm{y_i}(\bar{y})}{ \logm{y_i}{(y_j)} } \\
    &\geq \sumn{i} \sumn{j} w_{ij} \frac{1}{2} \Big(d^2(y_i,\bar{y}) - d^2(y_j,\bar{y}) + C_2 d^2(y_i,y_j) \Big) \\
    &=\frac{C_2}{2} \sumn{i} \sumn{j} w_{ij} d^2(y_i,y_j). 
\end{align*}
To bound the term $T_2$, we apply Jensen's inequality to get
\begin{equation*}
    T_2 \le C_1 \sumn{i}\sumn{j} w_{ij} d^2(y_i,y_j). 
\end{equation*}
By combining bounds on $T_1$ and $T_2$, we obtain
\begin{align*}
\sumn{i} d^2(x_i(s),\bar{y}) &\leq \sumn{i} d^2(y_i,\bar{y}) \\ 
& -(sC_2-s^2C_1) \sumn{i}\sumn{j} w_{ij} d^2(y_i,y_j). 
\end{align*}
Under the condition $s\le \frac{C_2}{C_1}$, the term $sC_2 -s^2C_1$ is non-negative. Hence, we can use the lower bound in Lemma \ref{lem:dec_ineq} on $\sumn{i}\sumn{j} w_{ij} d^2(y_i,y_j)$ to get 
\begin{equation*}
\sumn{i} d^2(x_i(s),\bar{y}) \leq \Big(1 - q(s)\frac{2(1-\sigma_2(W))}{(1+C_4D^2)^2} \Big) \sumn{i} d^2(y_i,\bar{y}), 
\end{equation*}
where $q(s) = sC_2-s^2C_1$. The optimal choice of $s$ can be computed by maximizing the quadratic term $sC_2-s^2C_1$ with $s=\frac{C_2}{2C_1}$. As a result, we obtain
\begin{equation*}
\sumn{i} d^2(x_i(s),\bar{y}) \leq \Big(1 -\frac{C_2^2(1-\sigma_2(W))}{2C_1(1+C_4D^2)^2} \Big)\sumn{i} d^2(y_i,\bar{y}). 
\end{equation*}
Let $\rho:=1 -\frac{C_2^2(1-\sigma_2(W))}{2C_1(1+C_4D^2)^2}$ denote the linear variance reduction coefficient. Then, we obtain
\begin{align*}
    \varset{x_i(s)}&\le \frac{1}{n}\sumn{i} d^2(x_i(s),\bar{y}) \\ 
   \vphantom{\frac{1}{n}\sumn{i}} & \leq  \frac{\rho}{n} \sumn{i} d^2(\bar{y},y_i)\\
    \vphantom{\frac{1}{n}\sumn{i}} &= \rho \varset{y_i}.
\end{align*}
\end{proof}
\subsection{Proof of Lemma \ref{lem:net_error}}
\begin{proof}
 We begin with a property of projection on g-convex sets, where we have that $\innerprod{\logm{\projman{\mansubset}(y)}(y)}{\logm{\projman{\mansubset}(y)}(x)} \le 0, \forall x \in \mansubset$ and $\forall y\in \man\setminus\mansubset$ (see \cite{walter1974metric}, or Lemma 47 of \cite{wang2023online}). We consider a geodesic triangle $\Delta (y_{i,t+1},z_{i,t+1},x_{i,t})$ where $z_{i,t+1}=\expm{x_{i,t}}(-\eta g_{i,t})$ and $y_{i,t+1}=\projman{\mansubset}(z_{i,t+1})$. Using Assumption~\ref{assum:fun} and Lemma \ref{lem:riem_cos}, we obtain 
\begin{align}
\label{eq:var_update}
    (\eta L)^2 \geq \eta^2 \norm{g_{i,t}}^2 & \geq d^2(x_{i,t},z_{i,t+1}) \geq d^2(x_{i,t},y_{i,t+1}),
\end{align}
where the last inequality is due to the projection property. The next step is to introduce the relationship between $\varset{x_{i,t}}$ and $\varset{x_{i,t+1}}$ using Theorem~\ref{thm:consensus} as follows 
\begin{align*}
    \sqrt{\sumn{i} d^2(x_{i,t+1},\bar{x}_{t+1})} &\leq 
     \sqrt{\rho\sumn{i} d^2(y_{i,t+1},\bar{y}_{t+1})} \quad \nonumber \\
    &\leq  \sqrt{\rho\sumn{i} d^2(y_{i,t+1},\bar{x}_{t})} \quad \nonumber \\
    & \leq \sqrt{\rho\sumn{i} d^2(x_{i,t},\bar{x}_{t})} \nonumber \\
    &+ \sqrt{\rho\sumn{i}d^2(x_{i,t},y_{i,t+1})} \quad \nonumber \\
    & \leq \sqrt{\rho\sumn{i} d^2(x_{i,t},\bar{x}_{t})} + \sqrt{\rho n} \eta L,
\end{align*}
where we applied triangle inequality to derive the last line and then applied the bound in \eqref{eq:var_update}.
We now established the relationship between the variance of consecutive iterations. So, we can compute the upper bound on the variance at time $t$ by recursively applying this result. We also assume that the initial point of all agents is the same for simplicity, so $d^2(x_{i,1},\bar{x}_1)=0$. Then,
\begin{equation}
    d(x_{i,t+1},\bar{x}_{t+1}) \le \sqrt{\sumn{i} d^2(x_{i,t+1},\bar{x}_{t+1})} \le \frac{\sqrt{\rho n}\eta L}{1-\sqrt{\rho}}.
\end{equation}
To simplify the last inequality, we can use $\sqrt{\rho}\leq \frac{1+\rho}{2}$ and $1-\sqrt{\rho}\geq \frac{1-\rho}{2} $, which hold for any $\rho\in [0,1]$. Hence, we obtain $d(x_{i,t+1},\bar{x}_{t+1}) \le \frac{1+\rho}{1-\rho} \sqrt{n}\eta L \le \frac{2\sqrt{n}\eta L}{1-\rho} $.

\end{proof}

\subsection{Proof of Theorem \ref{thm:staticregret}}
\label{proof:regretfull}
\begin{proof}

Let $x^*\in \mansubset$ be a minimizer of $\sum_{t=1}^T f_t(x)$. We can decompose the regret term into two terms as follows
\begin{align*}
     \regstafull&= \frac{1}{n} \sum_{t=1}^T \sumn{i} f_t(x_{i,t}) - \sum_{t=1}^T f_t(x^*)\\
    &= \frac{1}{n} \sum_{t=1}^T \sumn{i} f_t(x_{i,t}) - \frac{1}{n} \sum_{t=1}^T \sumn{i} f_t(\bar{x}_t)\nonumber \\
    &+ \underbrace{\frac{1}{n} \sum_{t=1}^T \sumn{i} f_t(\bar{x}_t)-\frac{1}{n} \sum_{t=1}^T \sumn{i} f_{i,t}(x_{i,t})}_{T_3}  \\
    &+ \underbrace{\frac{1}{n} \sum_{t=1}^T \sumn{i} f_{i,t}(x_{i,t}) -\sum_{t=1}^T f_t(x^*)}_{T_4}.
\end{align*}
To bound $T_3$, we can use the fact that the functions $f_{i,t}$ and $f_t = \frac{1}{n}\sumn{i}f_{i,t}$ are geodesically $L$-Lipschitz continuous. Hence, we have $f_t(x_{i,t})-f_t(\bar{x}_t) \le L d(\bar{x}_t,x_{i,t})$ and $f_{i,t}(x_{i,t})-f_{i,t}(\bar{x}_t) \le L d(\bar{x}_t,x_{i,t})$. Then, we can use Lemma \ref{lem:net_error} for the distance term. As a result, we obtain
\begin{equation*}
    T_3\le \frac{2L}{n} \sumn{i} \sumt d(\bar{x}_t,x_{i,t}) \le \eta T \frac{4\sqrt{n} L^2}{1-\rho}. 
\end{equation*}
To bound $T_4$, we use the g-convexity of $f_{i,t}$. Let $z_{i,t+1} = \expm{x_{i,t}}(-\eta~\rgrad f_{i,t}(x_{i,t} ))$. We have 
\begin{align}
\vphantom{\frac{1}{\eta}}f_{i,t}(x^*)-f_{i,t}(x_{i,t}) &\geq \innerprod{\rgrad f_{i,t}(x_{i,t})}{\logm{x_{i,t}}(x^*)}\nonumber\\
\vphantom{\frac{1}{\eta}}& = -\frac{1}{\eta} \innerprod{\logm{x_{i,t}}(z_{i,t+1})}{\logm{x_{i,t}}(x^*)}\nonumber\\
\vphantom{\frac{1}{\eta}}&\geq -\frac{1}{2\eta} (d^2(x_{i,t},x^*) -d^2(z_{i,t+1},x^*)\nonumber \\
&+ C_1 \eta^2L^2),
\end{align}
where in the last inequality we used Lemma \ref{lem:riem_cos} and the fact that $d(x_{i,t},z_{i,t+1}) \leq \eta L$. By rearranging the terms we obtain 
\begin{align*}
&\sumt f_{i,t}(x_{i,t})-f_{i,t}(x^*)\nonumber \\ 
&\leq \frac{1}{2\eta} \sumt  d^2(x_{i,t},x^*) -d^2(x_{i,t+1},x^*) + \sumt \frac{C_1}{2} \eta L^2 \\
&+ \frac{1}{2\eta} \sumt d^2(x_{i,t+1},x^*) -d^2(z_{i,t+1},x^*)\\
&\le \frac{D^2}{2\eta} + \frac{1}{2}C_1 T\eta L^2\\ 
&+ \frac{1}{2\eta}\underbrace{\sumt  d^2(x_{i,t+1},x^*) -d^2(y_{i,t+1},x^*)}_{T_{i,5}}\\
&+\frac{1}{2\eta}\underbrace{\sumt d^2(y_{i,t+1},x^*) -d^2(z_{i,t+1},x^*)}_{T_6}.
\end{align*}
For the term $T_{i,5}$, we use Lemma~\ref{lem:riem_cos}. For any $p\in \mansubset$, it holds that
\begin{align}
   \vphantom{\sumn{j}} &d^2(x_{i,t+1},p)-d^2(y_{i,t+1},p) \le C_1 d^2(x_{i,t+1},y_{i,t+1}) \nonumber\\
    \vphantom{\sumn{j}}&-2\innerprod{\logm{y_{i,t+1}}(x_{i,t+1})}{\logm{y_{i,t+1}}(p)}\nonumber\\
    \vphantom{\sumn{j}}&\le C_1 d^2(x_{i,t+1},y_{i,t+1}) \nonumber\\
    \vphantom{\sumn{j}}&-2s \sumn{j}w_{ij} \innerprod{\logm{y_{i,t+1}}(y_{j,t+1})}{\logm{y_{i,t+1}}(p)}.\nonumber 
\end{align}
Let us define $A_t:=\sumn{i} (d^2(x_{i,t+1},x^*) -d^2(y_{i,t+1},x^*))$ and choose $p=x^*$ in above. Summing above over $i$ and applying Lemma~\ref{lem:riem_cos} again gives
\begin{align}
A_t &\le  C_1 \sumn{i} d^2(x_{i,t+1},y_{i,t+1}) \nonumber\\
&-2s \sumn{i}\sumn{j}w_{ij} \innerprod{\logm{y_{i,t+1}}(y_{j,t+1})}{\logm{y_{i,t+1}}(x^*)}\nonumber\\
&\le C_1 \sumn{i} d^2(x_{i,t+1},y_{i,t+1}) -s \sumn{i}\sumn{j} w_{ij} \Big(d^2(y_{i,t+1},x^*) \nonumber\\
&-d^2(y_{j,t+1},x^*)+C_2d^2(y_{i,t+1},y_{j,t+1})\Big) \nonumber\\
&\le C_1 \sumn{i} d^2(x_{i,t+1},y_{i,t+1}) \nonumber\\
&-s C_2 \sumn{i}\sumn{j} w_{ij} d^2(y_{i,t+1},y_{j,t+1})\nonumber\\
&\le (C_1 s^2 -sC_2) \sumn{i} \sumn{j} w_{ij}d^2(y_{i,t+1},y_{j,t+1}),
\end{align}
where the last line is derived based on the update of $x_{i,t+1}$. Since we have $s=\frac{C_2}{2C_1} \le \frac{C_2}{C_1}$, we get that $\sumn{i}T_{i,5} = \sumt{A_t} \le 0$. In the next step, we bound the term $T_6$ with Lemma~\ref{lem:projerror}. With the choice of $\eta$ such that $\eta L \le D$, we have 
\begin{align}\label{eq:T6}
    \frac{1}{2\eta}T_6 &\le  C_7\sumt \frac{ \eta }{2}\norm{g_{i,t}}^2\le \frac{ \eta }{2} T L^2 C_7,
\end{align}
where $C_7:=\max(0,-g_2(\kmax,2D))$. Summing the bounds on $T_{i,5}$ and $T_6$  gives the following result 
\begin{align}
T_4&=\frac{1}{n} \sumn{i}\sumt f_{i,t}(x_{i,t})-f_{i,t}(x^*)\nonumber\\
 &\le \frac{D^2}{2\eta} + \frac{1}{2}\eta T L^2 (C_1+C_7).
\end{align}
Lastly, we sum the bounds on $T_3$ and $T_4$ to upper bound the regret term as
\begin{equation}
    \regstafull \le \frac{D^2}{2\eta} +\eta T \Big(\frac{4\sqrt{n} L^2}{1-\rho} +  \frac{1}{2} L^2 (C_1+C_7)\Big). \nonumber
\end{equation}
By choosing $\eta = \frac{D}{L\sqrt{T}} (\frac{8\sqrt{n}}{1-\rho} + C_1+C_7)^{-\frac{1}{2}} $, we obtain the static regret bound as follows
\begin{equation}
    \regstafull \le C_5DL\sqrt{T}, 
\end{equation}
where $C_5 := \sqrt{\frac{8\sqrt{n}}{1-\rho} + C_1+C_7 }$. 
\end{proof}

\begin{lemma}[Lemma 21 of \cite{wang2023online}]
\label{lem:projerror}
Suppose $\mansubset \subseteq M$ with radius $D < \frac{\pi}{2\sqrt{\kmax}}$. Let us define the iterates  $z_{i,t+1}=\expm{x_{i,t}}(-\eta g_{i,t})$ and $y_{i,t+1}=\projman{\mansubset}(z_{i,t+1})$ with $\norm{\eta g_{i,t}}\le D$. Then, it holds that 
\begin{align*}
\frac{1}{2\eta} \sumt&  d^2(y_{i,t+1},x^*)-d^2(z_{i,t+1},x^*) \nonumber\\
&\le C_7 \sumt \frac{\eta}{2}\norm{g_{i,t}}^2.
\end{align*}    
\end{lemma}

\subsection{Auxiliary Lemmas for the Bandit Setting}
In this section, we present a few lemmas that are useful for the analysis of the bandit feedback setting. For this section, let $\man$ be a Riemannian manifold whose sectional curvatures are in the interval $[\kmin,\kmax]$, and let $\kappa:=\max\{|\kmin|,|\kmax|\}$. All curvature-dependent constants in the lemmas are provided explicitly in Appendix \ref{app:constant}.

\begin{lemma}[Proposition A.3 \cite{criscitiello2023accelerated}]
\label{lem:pt_dexp}
Consider $u=\logm{x}(y) \in \tansp{x}$, the geodesic $\gamma(t)=\expm{x}(tu)$ and $y=\gamma(1)$. If $\gamma$ is defined and has no interior conjugate point on the interval $[0,1]$, then $\forall v\in \tansp{x}$ 
\begin{equation*}
    \norm{\ptg{x}{y}[v]-d \expm{x}(u) [v]}\le \kappa g_3(\kmin,\norm{u})\norm{u}^2\norm{v}.
\end{equation*}

\end{lemma}

\begin{lemma}
\label{lem:dlog-Pt}
Consider $x,y \in \man$ and the geodesic $\gamma(t)$ such that $x=\gamma(0)$ and $y=\gamma(1)$. Then, $\forall v \in \tansp{y}$
\begin{equation*}
    \norm{\ptg{y}{x}[v]-d\logm{x}(y)[v]}\le \kappa g_4(\kmax,d(x,y)) d(x,y)^2 \norm{v}.
\end{equation*}
\end{lemma}

\begin{proof}
We can decompose $v=v^{\mathrm{tan}}+v^{\mathrm{nor}}$ into its tangential and normal component onto the direction of $\logm{y}(x)$. Due to equality of $d\logm{x}(y)$ and $\ptg{y}{x}$ on the radial direction, we have $d\logm{x}(y)[v^{\mathrm{tan}}]=\ptg{y}{x} [v^{\mathrm{tan}}]$. By using the linearity of both operators we have $\ptg{y}{x}[v]-d\logm{x}(y)[v]=\ptg{y}{x}[v^{\mathrm{nor}}]-d\logm{x}(y)[v^{\mathrm{nor}}]$. Hence, it is sufficient to prove the lemma for the case that $v$ is a normal vector to the direction of $\logm{y}(x)$, which we assume in the sequel without loss of generality.

Let us define the geodesic variation $c(t,s)=\expm{x}(t(\logm{x}(\expm{y}(sv))))$ and Jacobi field $J(t)=\Dds c(t,s)|_{s=0}$. We have $J(1)=v$, $J(0)=0$ and $J'(0)=d \logm{x}(y)[v]$. Since the tangential component of  $J$ is $0$ at $t=0$ and $t=1$, $J$ is a normal vector field. We also define $V(t)$ in $\tansp{x}$ as $V(t)=\ptg{\gamma(t)}{x}[J(t)]$, such that $V(1)=\ptg{y}{x}[v]$, $V(0)=0$ and $V'(0)=J'(0)$. We want to bound the term $E:=\norm{V(1)-V(0)-V'(0)}=\norm{\int_{0}^{1}(1-t) V''(t) dt}$. Since $V(t)$ is parallel transported version of $J(t)$, its second derivative is $V''(t)=-\ptg{\gamma(t)}{x}[R(J(t),\gamma'(t))\gamma'(t)]$, where $R$ is the Riemann curvature tensor. We now have
\begin{align*}
E&\le \int_{0}^{1} (1-t) \norm{R(J(t),\gamma'(t))\gamma'(t)}dt\\ 
&\le \kappa d(x,y)^2\int_{0}^{1} (1-t)\norm{J(t)}dt,   
\end{align*} where we use Jacobi equation.

We bound $\norm{J(t)}$ by using Theorem 6.5.1 in \cite{jost2008riemannian} and we get $\norm{J(t)}\le \frac{s_{\kmax}(t,d(x,y))}{s_{\kmax}(d(x,y))}\norm{J(1)}$, where $s_{\kmax}$ is a curvature-dependent constant. As a result, we have $E\le g_4(\kmax,d(x,y)) \kappa d(x,y)^2 \norm{v}$, where $g_4(K,r)=\int_{0}^{1}(1-t) \frac{s_K(t,r)}{s_K(r)}dt$, and it is explicitly given in Appendix \ref{app:constant}.
\end{proof}

\begin{lemma}
\label{lem:tri_distortion}
Consider $x,y,z \in \man$ and suppose that their pairwise distances are less than $D\le \rcx$. Then 
\begin{equation*}
    \norm{\logm{x}(y)-\logm{x}(z) - \ptg{z}{x}[\logm{z}{(y)}]} \le C_9d(y,z)D^2.
\end{equation*}

\end{lemma}

\begin{proof}

Let us define $\gamma(t)=\expm{z}(t\logm{z}(y))$ and $w(t)=\logm{x}(\gamma(t))$. We want to upper bound the norm of $w(1)-w(0)-\ptg{z}{x}[\logm{z}(y)]=\int_0^1 w'(t)dt- \ptg{z}{x}[\logm{z}(y)]$. We can bound this as a summation of the two following error terms:
\begin{align*}
    E_1&=\int_0^1 \norm{ d \logm{x}(\gamma(t))[\gamma'(t)]-\ptg{\gamma(t)}{x}[\gamma'(t)]}dt\\
    E_2&=\int_0^1 \norm{\ptg{\gamma(t)}{x}[\gamma'(t)]-\ptg{z}{x}[\gamma'(0)]}dt,
\end{align*}
where $E_1$ is a radial logarithmic distortion and $E_2$ comes from holonomy error.

To bound $E_1$, we use Lemma \ref{lem:dlog-Pt} and we have $E_1\le C_4 D^2 d(y,z)$ where $C_4=\kappa g_7(\kmax,D)$.

To bound $E_2$, we use Lemma 6 in \cite{sun2019escaping} so that  $\norm{\ptg{\gamma(t)}{x}[\ptg{z}{\gamma(t)}[\gamma'(0)]]-\ptg{z}{x}[\gamma'(0)]} \le C_{10} \norm{\gamma'(0)}d(\gamma(t),z)d(x,\gamma(t))\le C_{10} d(y,z)D^2 $.

Summing up the two errors, we obtain $E_1+E_2\le (C_{10}+C_4)d(y,z)D^2$. Define $C_9:=C_{10}+C_4$ and the proof is complete.
\end{proof}

We next discuss the proofs related to the Riemannian two-point bandit setting. Let us drop the time index $t$ and agent index $i$ for simplicity. Let $u$ be uniformly distributed on the unit sphere in the tangent space of $x$, i.e.,  $u \sim \un (\sptansp{x}{1})$. Then, define the gradient estimator $\gdel(x)$ such that
\begin{equation}
    \gdel (x) = \frac{d}{2\delta} (f(\expm{x}(\delta u)) - f(\expm{x}(-\delta u)))u. 
\end{equation}

By Stokes' theorem $\mathbb{E}_u [\gdel(x)] = \int_{\balltansp{x}{\delta}} \nabla h_x(u) d\mu(u) $ where $h_x :\tansp{x} \to \mathbb{R}$ takes the form $h_x(u)=f(\expm{x}(u))$ and $\mu$ denotes a uniform measure on $\balltansp{x}{\delta}$. Let $v \in \tansp{x}$ and we have $\innerprod{\nabla h_x(u) }{v}=dh_x(u)[v]=df(\expm{x}(u))[d \expm{x}(u) [v]]=\innerprod{\rgrad{f}(\expm{x}(u))}{d \expm{x}(u) [v]}$.
\begin{align}\label{eq:19}
&\innerprod{\mathbb{E}_u [\gdel(x)]}{v} = \nonumber\\
&\int_{\balltansp{x}{\delta}} \innerprod{\rgrad{f}(\expm{x}(u))}{d \expm{x}(u) [v]}d\mu(u).  
\end{align}

\subsection{Proof of Lemma~\ref{lem:subconvexity}}
We want to find the infimum over $x,y\in (1-\tau)\mansubset$ of $\fdel(y)-\fdel(x) -\innerprod{\eg}{\logm{x}{(y)}}$ to prove g-subconvexity of $\fdel$.

Let us pick a vector $u\in \tansp{x}$ such that $\norm{u}\le \delta$ and define two points $p_u=\expm{x}(u)$ and $q_u=\expm{y}(\ptg{x}{y}[u])$. Due to g-convexity of $f$ we have $f(q_u) - f(p_u) \geq \innerprod{\rgrad{f(p_u)}}{\logm{p_u}(q_u)} $. Taking expectation over $u \sim \un (\balltansp{x}{\delta})$ and using \eqref{eq:19}, we get

\begin{align*}
    &\fdel(y)-\fdel(x) -\innerprod{\eg}{\logm{x}{(y)}}\\& \geq \int_{\balltansp{x}{\delta}} \innerprod{\rgrad{f(p_u)}}{\logm{p_u}(q_u)}  d\mu(u)\\
     &- \int_{\balltansp{x}{\delta}} \innerprod{\rgrad{f (p_u)}} {d \expm{x}(u) [\logm{x}(y)]} d\mu(u)\\
    & \geq - \int_{\balltansp{x}{\delta}} \norm{\rgrad{f (p_u)}}\\
    &~~~~~~~~~~\times \norm{\logm{p_u}(q_u) - d \expm{x}(u) [\logm{x}(y)]  }d\mu(u) \\
    & \geq -L \underset{u \in \balltansp{x}{\delta} }{\max} \norm{\logm{p_u}(q_u) - d \expm{x}(u) [\logm{x}(y)]}. 
\end{align*}
The problem is then reduced to finding the maximum value of $\norm{\logm{p_u}(q_u) - d \expm{x}(u) [\logm{x}(y)]}$. We can bound this term by adding and subtracting parallel transport of $\logm{x}(y)$, such that we have two error terms $E_1:=\norm{\logm{p_u}(q_u) - \ptg{x}{p_u}[\logm{x}(y)]}$ and $E_2:=\norm{\ptg{x}{p_u}[\logm{x}(y)] - d\expm{x}(u) [\logm{x}(y)]}$. 

For the error term $E_1$ we can use Lemma \ref{lem:tri_distortion} twice as follows 
\begin{align*}
&\logm{p_u}(q_u) -\ptg{x}{p_u}[\logm{x}(y)]\\
&=\logm{p_u}(q_u)-\logm{p_u}(x)-\ptg{x}{p_u}[\logm{x}(q_u)]\\   
&+\ptg{x}{p_u}[\logm{x}(q_u) -\logm{x}(y) - \ptg{y}{x}[\logm{y}(q_u)]]\\
&+\logm{p_u}(x)+\ptg{x}{p_u}[\ptg{y}{x}[\logm{y}(q_u)]].
\end{align*}
In the above decomposition, the last line is exactly zero since it is equal to $-\ptg{x}{p_u}[\logm{x}(p_u)]+\ptg{x}{p_u}[\ptg{y}{x}[\logm{y}(q_u)]]=\ptg{x}{p_u}[-u+u]=0$.  

The first two lines arise from geodesic triangles $\Delta xp_uq_u$ and $\Delta xq_uy$. We have that $d(x,p_u)=d(y,q_u)\leq\norm{u}\leq \delta$, and all points belong to $\mansubset$ with a diameter at most $D$. For the first line, we can apply $\ptg{p_u}{q_u}$, which preserves the norm, so we obtain an upper bound with two terms $\norm{-\logm{q_u}(p_u)-\ptg{p_u}{q_u}[\logm{p_u}(x)]+\logm{q_u}(x)}$ and $\norm{-\ptg{p_u}{q_u}[\ptg{x}{p_u}[\logm{x}(q_u)]]+\ptg{x}{q_u}[\logm{x}(q_u)]}$. We can bound these two terms by $C_9\delta D^2$ and $C_{10}d(x,p_u)d(p_u,q_u)d(x,q_u)\le C_{10}\delta D^2$ using Lemma \ref{lem:tri_distortion} and Lemma 6 in \cite{sun2019escaping}, respectively. For the second line we can also apply Lemma \ref{lem:tri_distortion} and we get $E_1\le 2C_9\delta D^2 + C_{10}\delta D^2 $. 

To bound $E_2$ we use Lemma \ref{lem:pt_dexp} and we obtain  $E_2\le  C_{11}\delta^2D$ where $C_{11}=\kappa g_6(\kmin,\delta)$. 

Summing up the error terms gives 
$$E_1+E_2 \le \delta (C_{11}\delta D+2C_9 D^2+C_{10}D^2)\le \delta C_{6}.$$ 
Hence, we have $$\fdel(y)-\fdel(x) -\innerprod{\mathbb{E}_u [\gdel(x)]}{\logm{x}{y}} \geq -\delta LC_6,$$ which proves that $\fdel$ is $\delta L C_6$ g-subconvex.

\subsection{Proof of Theorem \ref{thm:DOR2BAN}}
\begin{proof}
We upper bound $\mathbb{E}[\regstaban]$ by decomposing it to a summation of network error, subconvexity error, and projection error. Denote by $x^*_{\tau}$ the minimizer of the problem 
${\min}_{x\in \shrinkset}\sum_{t=1}^T f_t(x)$, and recall that $\fdel_t(x):=\int f_t(\expm{x}(\delta u)) dp(u)$ is the smoothed version of $f_t(x)$ with $dp(u)$ denoting a uniform measure on $\sptansp{x}{1}$. Then,
\begin{align*}
    &\mathbb{E}[\regstaban] =\frac{1}{n} \sumn{i}\sum_{t=1}^T \mathbb{E}\left[\frac{f_t(x_{i,t,1}) + f_t(x_{i,t,2})}{2}-f_t(x^*)\right]\\
    &=\mathbb{E}\left[\frac{1}{n} \sumn{i}\sum_{t=1}^T \frac{f_t(x_{i,t,1}) + f_t(x_{i,t,2})}{2} - f_t(x_{i,t})  \right] \\
    &+ \mathbb{E}\left[\frac{1}{n} \sumn{i}\sumt f_t(x_{i,t})-\fdel_t (x_{i,t})\right] \\
    &+ \mathbb{E}\left[\sumt \fdel_t (x_{\tau}^*)- f_t(x_{\tau}^*) \right] \\
    &+\mathbb{E}\left[\frac{1}{n} \sumn{i}\sumt \fdel_t(x_{i,t})-\fdel_t (x_{\tau}^*)\right]  \\
    &+ \mathbb{E}\left[\sumt f_t(x_{\tau}^*) - f_t(x^*) \right].
 \end{align*}
Since $d(x_{i,t,j},x_{i,t})\leq \delta$ for $j=1,2$, Lipschitz conditions of $f_t$ lead to
\begin{equation*}
    \begin{cases}
    f_t(x_{i,t,1})-f_t(x_{i,t}) \le \delta L\\
    f_t(x_{i,t,2})-f_t(x_{i,t}) \le \delta L\\
    f_t(x_{i,t}) - \fdel_t(x_{i,t}) \le \delta L\\
    \fdel_t(x_{\tau}^*) - f_t(x_{\tau}^*) \le \delta L.   
\end{cases}
\end{equation*}
Also, by the definition of the shrinking set $\shrinkset$, there exists a point $p$ such that $\expm{p}((1-\tau)\logm{p}(x^*))\in (1-\tau)\mansubset$. By using the geodesic convexity of $f_t$ we obtain
\begin{align}
    \sumt f_t(x_{\tau}^*) &\le \sumt f_t(\expm{p}((1-\tau)\logm{p}(x^*))) \nonumber\\ 
    &\le (1-\tau)\sumt f_t(x^*)+ \tau \sumt f_t(p) \nonumber \\
    &= \sumt f_t(x^*) +\tau \sumt f_t(p)-f_t(x^*) \nonumber \\
    &\le \sumt f_t(x^*) + \tau DLT.
\end{align}
Thus, we can write the regret of the algorithm on the functions $f_{i,t}$ over the set $\mansubset$ in terms of the regret of the algorithm on the functions $\fdel_{i,t}$ over the set $(1-\tau)\mansubset$ as follows
\begin{align*}
    \mathbb{E}[\regstaban] &=\frac{1}{n} \sumn{i}\sum_{t=1}^T \mathbb{E}\left[\frac{f_t(x_{i,t,1}) + f_t(x_{i,t,2})}{2} -f_t(x^*)\right]\\
    &\le \mathbb{E}\left[\frac{1}{n} \sumn{i}\sumt \fdel_t(x_{i,t})-\fdel_t (x_{\tau}^*)\right] \\
    & + 3\delta LT + \tau DLT.
\end{align*}
To handle the first term above, we use the same decomposition as $T_3$ and $T_4$ in Proof of Theorem \ref{thm:staticregret}. Since $\norm{\gdel_{i,t}(x)}\le dL $, we have $\mathbb{E}[\norm{\rgrad{\fdel_{i,t}(x)}}] \le dL$. Hence, due to the projection error of the shrinking set, the network error bound (result of Lemma \ref{lem:net_error}) in the bandit setting changes to $\frac{2d\sqrt{n}\eta L}{1-\rho}+\frac{2\sqrt{n} \tau D}{1-\rho}$, and the bound for the term corresponding to $T_3$ will be $\eta T \frac{4\sqrt{n} (dL)^2}{1-\rho} +\tau D T \frac{4\sqrt{n} dL}{1-\rho} $. 

To handle the term corresponding to $T_4$, we need to bound $\frac{1}{n}\sumt \sumn{i} \fdel_{i,t}(x_{i,t})-\fdel_{i,t}(x_{\tau}^*)$. We use subconvexity property of $\fdel_{i,t}(x)$ and bounded projection error for the set $(1-\tau)\mansubset$. Suppose that $\fdel_{i,t}$ is $\lambda_1$ g-subconvex and the projection operator satisfies $d^2(\projman{\shrinkset}(x),y) - d^2(x,y)\le \lambda_2$ for all $x,y \in \mansubset$. 
Lemma \ref{lem:subconvexity} shows that $\lambda_1=\delta LC_6$ and since
\begin{equation}
\label{eq:projection}
    d^2(\projman{(1-\tau)\mansubset}(x),y) - d^2(x,y) \le 2D d(\projman{(1-\tau)\mansubset}(x),x)\le 2\tau D^2,
\end{equation}
Equation \eqref{eq:projection} shows that $\lambda_2 = 2\tau D^2$.  We can then define $q_{i,t+1}=\expm{x_{i,t}}(-\eta\gdel_{i,t})$, $z_{i,t+1}=\projman{\mansubset}(q_{i,t+1})$ and $y_{i,t+1}=\projman{(1-\tau)\mansubset}(z_{i,t+1})$ to write
\begin{align*}
    &\mathbb{E}[\fdel_{i,t}(x_{i,t}) - \fdel_{i,t}(x_{\tau}^*)] \le \mathbb{E}[\innerprod{-\gdel_{i,t}}{\logm{x_{i,t}}(x_{\tau}^*)}] +\lambda_1 \\
    &\le \frac{1}{2\eta}\big(\mathbb{E}[d^2(x_{i,t},x_{\tau}^*) - d^2(q_{i,t+1},x_{\tau}^*)] + \eta^2 C_1 (dL)^2 \big) + \lambda_1\\
    &\le \frac{1}{2\eta}\big(\mathbb{E}[d^2(x_{i,t},x_{\tau}^*) - d^2(z_{i,t+1},x_{\tau}^*)] + \eta^2C_8 (dL)^2 \big) + \lambda_1 \\
    &\le \frac{1}{2\eta}\mathbb{E}[d^2(x_{i,t},x_{\tau}^*) - d^2(y_{i,t+1},x_{\tau}^*)]  + \lambda_1 +\frac{\eta^2C_8 (dL)^2+\lambda_2}{2\eta}\\
    &\le \frac{1}{2\eta}\mathbb{E}[d^2(x_{i,t},x_{\tau}^*) - d^2(x_{i,t+1},x_{\tau}^*)]  + \lambda_1 +\frac{\eta^2C_8 (dL)^2+\lambda_2}{2\eta} \\
    & + \frac{1}{2\eta}\mathbb{E}[d^2(x_{i,t+1},x_{\tau}^*) - d^2(y_{i,t+1},x_{\tau}^*)].
\end{align*}
As we showed in Section \ref{proof:regretfull}, the summation of the last term over agents is less than $0$, i.e., $\sumn{i}d^2(x_{i,t+1},x_{\tau}^*) - d^2(y_{i,t+1},x_{\tau}^*) \le 0$. Hence, summation over $t$ gives 
\begin{align*}
    \mathbb{E}[&\regstaban]\le \frac{D^2}{2\eta} + \eta T \frac{4\sqrt{n} (dL)^2}{1-\rho} + \tau D T \frac{4\sqrt{n} dL}{1-\rho}  \\ 
    &+ \eta T\frac{C_8(dL)^2}{2} +\lambda_1 T + \frac{\lambda_2 T}{2\eta}  + 3\delta LT + \tau DLT \\
    &\le \frac{D^2}{2\eta} + \eta T \Big(\frac{4\sqrt{n} (dL)^2}{1-\rho} + \frac{C_8(dL)^2}{2}\Big) \\ 
    &+\delta T (3L+LC_6) + \tau T \Big( DL + \frac{D^2}{\eta} + \frac{4\sqrt{n}DdL}{1-\rho}\Big).
\end{align*}

The last step is to define the shrinkage coefficient $\tau$ in terms of $\delta$. Throughout, we restrict $\tau \leq \frac{1}{2}$ as a conservative upper bound. Suppose that there exists a point $p\in \mansubset$, and two constants $0 \le r \leq R$ such that $B_{r}(p) \subseteq \mansubset \subseteq B_{R}(p)$ where $B_r(p)$ denotes the geodesic ball centered at $p$ with radius $r$. Let $x\in \mansubset$ be an arbitrary point and $y=\expm{p}((1-\tau)\logm{p}(x))$. For any $z\in B_{\delta}(y)$ with $d(y,z)\le \delta$ we want to have $z\in \mansubset$.

Let us define $q=\expm{x}(\frac{1}{\tau}\logm{x}(z))$. If $d(p,q)\le r$, then $z\in \mansubset$, since $\mansubset$ is geodesically convex and $z$ belongs to the geodesic connecting $q,x\in \mansubset$. Notice that
\begin{align*}
    \vphantom{\frac{1}{2}}d(p,q) &\le \norm{\logm{x}(p)-\logm{x}(q)}(1+C_{12}(2R+r)^2)\\
    \vphantom{\frac{1}{2}}&= \frac{\norm{\logm{x}(y)-\logm{x}(z)}}{\tau}(1+C_{12}(2R+r)^2)\\
    \vphantom{\frac{1}{2}}&\le \frac{d(y,z)}{\tau}(1+C_{13} ((R+r)/2)^2 )(1+C_{12}(2R+r)^2)\\
    \vphantom{\frac{1}{2}}&\leq\frac{\delta}{\tau} (1+C_{13} ((R+r)/2)^2 )(1+C_{12}(2R+r)^2),
\end{align*}
where $C_{12}$ and $C_{13}$ are defined in Appendix \ref{app:constant}. In the above, we applied Lemma \ref{lem:dist_lower} twice with pairwise distance bounds $2R+r$ and $\frac{R+r}{2}$, noting that $\frac{R+r}{2}\geq \tau(R+r)$ for geodesic triangles $\Delta xpq$ and $\Delta xyz$, respectively. Defining 
$$\theta:=(1+C_{13} ((R+r)/2)^2 )(1+C_{12}(2R+r)^2),$$ 
we can guarantee $d(p,q)\le r$ and feasibility of $z\in\mansubset$ by choosing $\tau = \frac{\delta \theta}{ r} $. Then, for every  $y \in \shrinkset$, the geodesic ball $\ballgeod{y}{\delta}$ lies in $\mansubset$.
Finally, taking 
$\delta=\frac{1}{T}$ we obtain 
\begin{align}
\label{eq:reg_2ban}
\mathbb{E}[&\regstaban] \le \frac{D^2}{2\eta} + \eta T \Big(\frac{4\sqrt{n} (dL)^2}{1-\rho} + \frac{C_8(dL)^2}{2} \Big) \nonumber \\
&+3L+LC_6+\frac{\theta}{r}\Big(DL+\frac{D^2}{\eta}+\frac{4\sqrt{n}DdL}{1-\rho}\Big).
\end{align}
The upper bound is in the form of $O(1 + \frac{1}{\eta}+\eta T)$, and with the choice of $\eta =O(T^{-1/2})$ the static regret of two-point bandit setting is $O(\sqrt{T})$.
\end{proof}

\subsection{Constant Terms}\label{app:constant}
In this section, we define the constant terms used in our paper. The first set of constants are defined as functions of other parameters:
\begin{align*}
&g_1(\kappa,r) := \begin{cases}
    \frac{\sqrt{-\kappa}r}{\tanh(\sqrt{-\kappa}r)} &  \text{ if } \kappa < 0\\
    1 &  \text{ if } \kappa \geq 0
    \end{cases} \\
&g_2(\kappa,r) := \begin{cases}
1 &  \text{ if } \kappa\leq 0\\
\sqrt{\kappa}r \cot(\sqrt{\kappa}r) &  \text{ if } \kappa > 
\end{cases}
\\
    &g_{3}(\kappa,r) :=\begin{cases}
    \frac{1}{6} & \text{ if } \kappa=0\\
    \frac{1}{\kappa r^2} (1-\frac{\sin{(r\sqrt{\kappa})}}{r\sqrt{\kappa}}) & \text{ if } \kappa>0 \\ 
    \frac{1}{-\kappa r^2 } (\frac{\sinh{(r\sqrt{-\kappa})}}{r\sqrt{-\kappa}}-1) & \text{ if } \kappa<0 \\
\end{cases}
\\
    &g_{4}(\kappa,r) :=\begin{cases}
    \frac{1}{6} & \text{ if } \kappa=0\\
    \frac{1}{\kappa r^2} (\frac{r\sqrt{\kappa}}{\sin{(r\sqrt{\kappa})}}-1) & \text{ if } \kappa>0 \\ 
    \frac{1}{-\kappa r^2 } (1-\frac{r\sqrt{-\kappa}}{\sinh{(r\sqrt{-\kappa})}}) & \text{ if } \kappa<0 \\
\end{cases}
\\
    &g_5(\kappa,r) :=\begin{cases}
    r & \text{if } \kappa=0\\
    \frac{1}{\sqrt{\kappa}} \sin{(\sqrt{\kappa}r)} & \text{ if } \kappa>0 \\ 
    \frac{1}{\sqrt{-\kappa}} \sinh{(\sqrt{-\kappa}r)} & \text{ if } \kappa<0 \\
\end{cases}
\\
&g_6(\kappa,r) := \begin{cases}
    \frac{1}{-\kappa r^2}(\frac{\sinh(r\sqrt{-\kappa})}{r\sqrt{-\kappa}}-1) & \text{ if } \kappa < 0\\
    \frac{1}{6} & \text{ if } \kappa \geq 0
    \end{cases}  \\
&g_7(\kappa,r) := \begin{cases}
\frac{1}{6} & \text{ if } \kappa\leq 0\\
\frac{1}{\kappa r^2} (\frac{r\sqrt{\kappa}}{\sin{(r\sqrt{\kappa})}}-1) & \text{ if } \kappa > 0
\end{cases}
\end{align*}

In the following, we define the absolute constants:
\begin{align*}
    \vphantom{\frac{1}{2}}C_1 &:= g_1(\kmin,D)\\
    \vphantom{\frac{1}{2}} C_2 &:= g_2(\kmax,D) \\
    \vphantom{\frac{1}{2}} C_3 &:= \kappa g_6(\kmin,D) \\
    \vphantom{\frac{1}{2}} C_4 &:= \kappa g_7(\kmax,D) \\
    \displaybreak[3]
    \vphantom{\frac{1}{2}}C_5 &:= \sqrt{\frac{8\sqrt{n}}{1-\rho} + C_1+C_7}\\
    \vphantom{\frac{1}{2}}C_6 &:= C_{11}\delta D+2C_9 D^2+C_{10}D^2 \\
    \vphantom{\frac{1}{2}}C_7 &:= \max(0,-g_2(\kmax,2D))\\ \vphantom{\frac{1}{2}}C_8&:=C_1+C_7 \\
    \vphantom{\frac{1}{2}}C_9&:=C_{10}+C_4  \\
    \vphantom{\frac{1}{2}}C_{11}&:=\kappa g_6(\kmin,\delta)\\
    \vphantom{\frac{1}{2}} C_{12} &:= \kappa g_6(\kmin,2R+r)\\
    \vphantom{\frac{1}{2}} C_{13} &:= \kappa g_7(\kmax,\frac{R+r}{2}) 
\end{align*}
For $C_{10}$ find the definition in Lemma 6 of \cite{sun2019escaping}.

\bibliographystyle{IEEEtran}
\bibliography{refs}

\begin{IEEEbiography}	[{\includegraphics[width=0.95in,height=1in,clip,keepaspectratio]{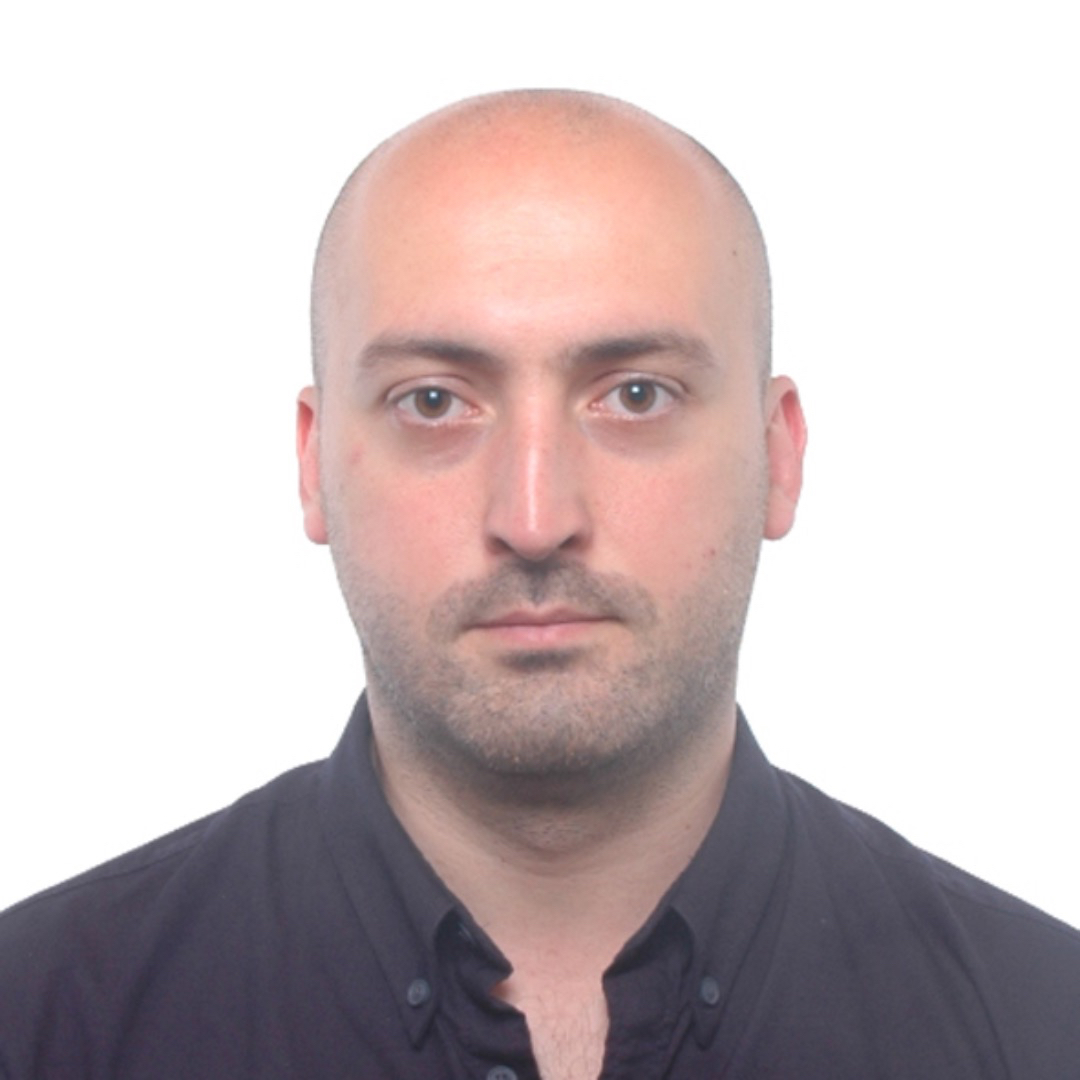}}] 
	{Emre Sahinoglu}   is currently a Ph.D. candidate in the Department of Mechanical and Industrial Engineering at Northeastern University. He received his B.S. degree in Electrical Engineering from Bilkent University. His research interests include machine learning and optimization, with emphasis on distributed and multi-agent systems. He is particularly interested in distributed optimization, online optimization, and Riemannian optimization.
\end{IEEEbiography}

\begin{IEEEbiography}
[{\includegraphics[width=1.35in,height=1.1in,clip,keepaspectratio]{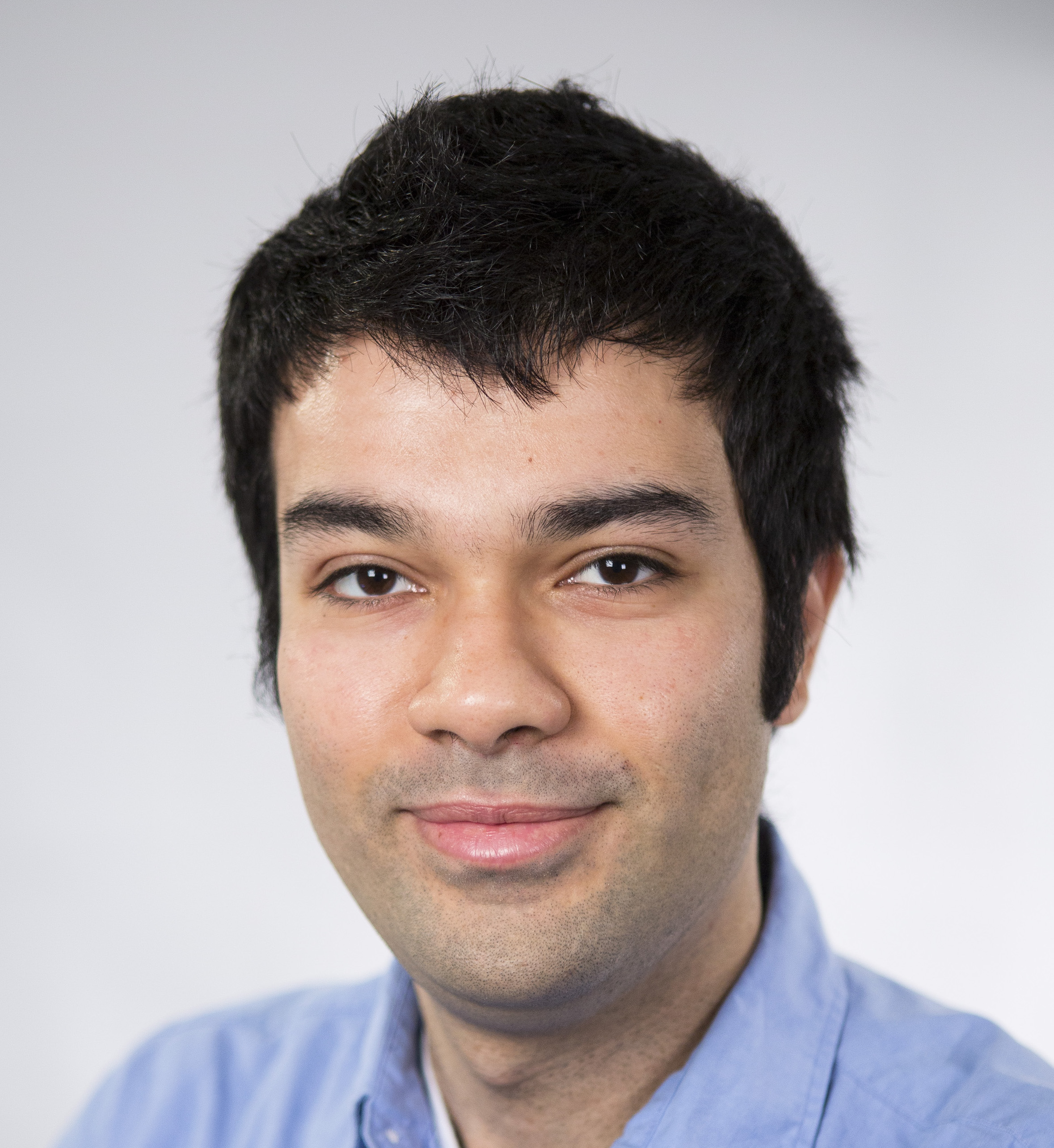}}] {Shahin Shahrampour} received the Ph.D. degree in Electrical and Systems Engineering, the M.A. degree in Statistics (The Wharton School), and the M.S.E. degree in Electrical Engineering, all from the University of Pennsylvania, in 2015, 2014, and 2012, respectively. He is currently an Assistant Professor in the Department of Mechanical and Industrial Engineering at Northeastern University. His research interests include machine learning, optimization, sequential decision-making, and distributed learning, with a focus on developing computationally efficient methods for data analytics. He is a Senior Member of the IEEE.
\end{IEEEbiography}

\end{document}